\documentclass[11pt, reqno,sumlimits]{amsart}

\usepackage{amssymb,amscd, epsfig, mathrsfs, xypic, amsmath,amsthm, dsfont, enumerate, txfonts, arydshln} 
\xyoption{all}
\newcommand{\xqedhere}[2]{%
  \rlap{\hbox to#1{\hfil\llap{\ensuremath{#2}}}}}

 \usepackage{amsthm}

\makeatletter
\newtheorem*{rep@theorem}{\rep@title}
\newcommand{\newreptheorem}[2]{%
\newenvironment{rep#1}[1]{%
 \def\rep@title{#2 \ref{##1}}%
 \begin{rep@theorem}}%
 {\end{rep@theorem}}}
\makeatother
 
\newtheorem*{thma}{Theorem A}
\newtheorem*{thmb}{Theorem B}
\newtheorem*{corc}{Corollary C}

\newtheorem{thm}{Theorem}[section] 
\newtheorem{cor}[thm]{Corollary}
\newtheorem{lem}[thm]{Lemma}  
\newtheorem{prop}[thm]{Proposition} 
\newtheorem{df-pr}[thm]{Definition-Proposition}

\theoremstyle{definition}
\newtheorem{defn}[thm]{Definition}
  
\newtheorem{rem}[thm]{Remark}

\newtheorem{exm}[thm]{Example}

\newtheorem{notation}[thm]{Notation} 
\newtheorem{question}[thm]{Question}

\newtheorem{construction}[thm]{Construction}

\newcommand{\RR}{{\mathbb R}}

\newcommand{\QQ}{{\mathbb Q}} 
\newcommand{\CC}{{\mathbb C}}

\newcommand{\ZZ}{{\mathbb Z}}

\newcommand{\sfa }{{\mathsf a}}
\newcommand{\sfb }{{\mathsf b }}

\newcommand{\sfe }{{\mathsf e}}
\newcommand{\sff }{{\mathsf f}}
\newcommand{\sfg }{{\mathsf g}}

\newcommand{\sfq }{{\mathsf q}}

\newcommand{\frakp}{{\mathfrak p}}

\newcommand{\frakr}{{\mathfrak  r}}

\newcommand{\frakt}{{\mathfrak  t}}

\newcommand{\bfe}{{\mathbf e}}

\newcommand{\calX}{{\mathcal X}}

\newcommand{\calZ}{{\mathcal Z}}

\newcommand{\sfX }{{\mathsf X}}
\newcommand{\sfW }{{\mathsf W}}

\newcommand{\sfU }{{\mathsf U}}
\newcommand{\sfT }{{\mathsf T}}

\newcommand{\sfR }{{\mathsf R}}

\newcommand{\sfN }{{\mathsf N}}
\newcommand{\sfM }{{\mathsf M}}
\newcommand{\sfL }{{\mathsf L}}
\newcommand{\sfK }{{\mathsf K}}

\newcommand{\sfG }{{\mathsf G}}

\newcommand{\sfD }{{\mathsf D}}

\newcommand{\CP}{{\mathbb C}{\mathbb P}}
\newcommand{\one}{{\bf 1}}
\newcommand{\surj}{\twoheadrightarrow}

\newcommand{\lan}{{\langle}}
\newcommand{\ran}{{\rangle}}

\newcommand{\inc}{\hookrightarrow}

\newcommand{\Hom}{\operatorname{Hom}}
\newcommand{\Ext}{\operatorname{Ext}}

\newcommand{\Ann}{\operatorname{Ann}}
\newcommand{\Lie}{\operatorname{Lie}}

\newcommand{\im}{{\operatorname{Im}}}

\newcommand{\U}{{\mbox{U}}}

\newcommand{\Tor}{{\operatorname{Tor}}}

\newcommand{\bideg}{{\operatorname{bideg\ }}}
\newcommand{\depth}{{\operatorname{depth}}}

\newcommand{\grade}{{\operatorname{grade}}}

\begin{document}
\title{Moment Angle Complexes and big Cohen-Macaulayness}
\author{Shisen Luo} \address{Cornell University, Mathematics Department, Ithaca, NY 14853} \email{sl943@cornell.edu}
\author{Tomoo Matsumura} \address{Algebraic Structure and its Applications Research Center, Department of Mathematical Sciences, KAIST, Daejeon, South Korea} \email{tomoomatsumura@kaist.ac.kr }%
\author{W. Frank Moore} \address{Department of Mathematics, Wake Forest University, Winston-Salem, NC 27109} \email{moorewf@wfu.edu}

\begin{abstract}  
Let $\calZ_K \subset \CC^m$ be the moment angle complex associated to a simplicial complex $K$ on $[m]$, together with
the natural action of the torus $\sfT=\U(1)^m$. Let $\sfG\subset \sfT$ be a (\emph{possibly disconnected}) subgroup and
$\sfR:=\sfT/\sfG$. Let $\ZZ[K]$ be the Stanley-Reisner ring of $K$ and consider $\ZZ[\sfR^*]:=H^*(B\sfR;\ZZ)$ as a subring of
$\ZZ[\sfT^*]:=H^*(B\sfT;\ZZ)$. We prove that $H_{\sfG}^*(\calZ_K;\ZZ)$ is isomorphic to
$\Tor_{\ZZ[\sfR^*]}^*(\ZZ[K],\ZZ)$ as a graded module over $\ZZ[\sfT^*]$. Based on this, we characterize the
surjectivity of $\iota^* :H_{\sfT}^*(\calZ_K;\ZZ) \to H_{\sfG}^*(\calZ_K; \ZZ)$ (i.e. $H^{odd}_{\sfG}(\calZ_K;\ZZ)=0$)
in terms of the vanishing of $\Tor_1^{\ZZ[\sfR^*]}(\ZZ[K],\ZZ)$ and discuss its relation to the freeness and the
torsion-freeness of $\ZZ[K]$ over $\ZZ[\sfR^*]$.  For various toric orbifolds $\calX$, by which we mean quasi-toric
orbifolds or toric Deligne-Mumford stacks, the cohomology of $\calX$ can be identified with $H_{\sfG}(\calZ_K)$ with
appropriate $K$ and $\sfG$ and the above results mean that $H^*(\calX;\ZZ) \cong \Tor_{\ZZ[\sfR^*]}^*(\ZZ[K],\ZZ)$ and
that $H^{odd}(\calX;\ZZ)=0$ if and only if $H^*(\calX;\ZZ)$ is the quotient $H_{\sfR}^*(\calX;\ZZ)$.
\end{abstract}
\maketitle
\section{{\bf Introduction}} 
The equivariant cohomology and the ordinary cohomology with $\ZZ$-coefficients of a ``\emph{compact smooth toric space}"
(including quasi-toric manifolds, complete smooth toric varieties) has been known by the work of Danilov
\cite{Danilov78}, Jurkievicz \cite{Jurkiewicz85}, and Davis-Januszkiewicz \cite{DavisJanuszkiewicz91}: the equivariant
cohomology is the Stanley-Reisner ring of the associated simplicial complex and the ordinary cohomology is the quotient
of the equivariant cohomology by linear relations.

The orbifold analogue of these spaces have been also introduced and
studied by several people, for example, Lerman-Tolman \cite{LT}, Borisov-Chen-Smith \cite{BCS}, Poddar-Sarkar
\cite{PoddarSarkar10}. The equivariant cohomology of these toric orbifolds with $\ZZ$-coefficients is also known
to be the associated Stanley-Reisner rings and the ordinary cohomology is the quotient of the equivariant cohomology over
$\QQ$-coefficients, c.f. \cite{Danilov78, Jurkiewicz85, BCS, PoddarSarkar10}. However the ordinary cohomology with
$\ZZ$-coefficients is hard to compute because it is not the quotient of the equivariant cohomology in general, for
example, the direct product of weighted projective spaces. The main theme of this paper is to characterize when the
ordinary cohomology is the quotient of the equivariant cohomology.

Our approach is to view previously mentioned toric orbifolds as \emph{quotient stacks} given by partial quotients of the
moment angle complexes and the \emph{cohomology of stacks} in the sense of \cite{Edidin10} (see also \cite{TolmanThesis,
  EdidinGraham98}). The \emph{moment-angle complex} $\calZ_K$ was introduced by Buchstaber and Panov in \cite{BP99} as a
disc-circle decomposition of the Davis-Januszkiewicz universal space associated to a simplicial complex $K$
\cite{DavisJanuszkiewicz91} where they introduced a quasi-toric manifold as a partial quotient of the moment angle
manifold associated to a simple polytope.

If $K$ is a simplicial complex on $[m]:=\{1,\cdots,m\}$, then $\calZ_K$
carries a natural action of the torus $\sfT:=\U(1)^m$. The quotient stack $[\calZ_K/\sfG]$ with an appropriate choice of
the subgroup $\sfG \subset \sfT$ can be used as a topological model to compute the cohomology of quasi-toric orbifolds
\cite{DavisJanuszkiewicz91, PoddarSarkar10}, symplectic toric orbifolds \cite{LT} and toric Deligne-Mumford stacks
\cite{BCS}, i.e. \emph{the ordinary cohomology of these toric orbifolds as stacks can be defined as the
  $\sfG$-equivariant cohomology $H^*_{\sfG}(\calZ_K;\ZZ)$.} Similarly the equivariant cohomology can be defined as
$H^*_{\sfT}(\calZ_K;\ZZ)$ which is isomorphic to the Stanley-Reisner ring $\ZZ[K]$ as quotient rings of
$\ZZ[\sfT^*]:=H^*(B\sfT;\ZZ)=\ZZ[x_1,\cdots, x_m]$. In Section \ref{bg}, we recall the constructions of those toric
orbifolds and the relation to the moment angle complexes to motivate our readers.

In Section \ref{TheoremFranz}, we start with proving

\begin{thma}
Let $\sfG \subset \sfT$ be a (possibly disconnected) subgroup
  and $\sfR:=\sfT/\sfG$.  There is an isomorphism of graded modules over $H^*(B\sfT;\ZZ)$,
\[
H_{\sfG}^*(\calZ_K;\ZZ) \cong \Tor_{\ZZ[\sfR^*]}^*(\ZZ[K];\ZZ).
\]
\end{thma}
Here $\ZZ[\sfR^*]:=H^*(B\sfR;\ZZ)=\ZZ[u_1,\cdots,u_n]$ is considered as a subring of $\ZZ[\sfT^*]$ so that $u_i$'s are
linear combinations of $x_j$'s.  It is worth noting that this theorem holds more generally. Namely, the theorem holds
for any a topological space $X$ with $\sfT$-action such that \emph{$C^*(E\sfT\times_{\sfT} X;\ZZ)$ is formal in the
  category of $H^*(B\sfT;\ZZ)$-modules up to homotopy} in the sense of \cite{Franz05i}. See Section \ref{TheoremFranz}
for the details.

Based on this isomorphism, we prove our characterization theorem:

\begin{thmb}
The following are equivalent: $(1)$
  $\Tor_1^{\ZZ[\sfR^*]}(\ZZ[K],\ZZ) = 0$$;$ $(2)$ $H^*_{\sfG}(\calZ_K;\ZZ)$ is isomorphic to the quotient of $\ZZ[K]$ by
  linear terms; $(3)$ $H^{odd}_{\sfG}(\calZ_K; \ZZ)=0$.
\end{thmb}

\vspace{2 mm}\noindent We will explain in Section \ref{ba} that, even though $\ZZ[K]$ might not be finitely generated
over $\ZZ[\sfR^*]$, the vanishing of $\Tor_1^{\ZZ[\sfR^*]}(\ZZ[K],\ZZ)$ has the usual meaning in terms of regular
sequences, i.e. $\Tor_1^{\ZZ[\sfR^*]}(\ZZ[K],\ZZ) = 0$ iff $u_1,\cdots, u_n$ form a $\ZZ[K]$-regular sequence. Thus we
say $\ZZ[K]$ is \emph{big Cohen-Macaulay over $\ZZ[\sfR^*]$} if (1) is satisfied.

By presenting a toric orbifold $\calX$ as $[\calZ_K/\sfG]$, we obtain the following immediate corollary:

\begin{corc}
If $\calX$ is a toric orbifold stack presented as $[\calZ_K/\sfG]$, then
\[
H^*(\calX;\ZZ) \cong \Tor_{\ZZ[\sfR^*]}^*(\ZZ[K];\ZZ).
\]
Furthermore, $H^*(\calX;\ZZ)$ is the quotient of Stanley-Reisner ring $\ZZ[K]$ if and only if one of the following
equivalent conditions holds: $(i)$ $H^{odd}(\calX;\ZZ)=0$; $(ii)$ $\Tor^{\ZZ[\sfR^*]}_1(\ZZ[K];\ZZ)=0$.
\end{corc}

\vspace{2 mm}\noindent For example, the cohomology of the weighted projective spaces as stacks are shown to be the
quotient of its equivariant cohomology, based on the computation exhibited in \cite{Holm08}. On the other hand, the
cohomology of a direct product of weighted projective spaces is not the quotient of its equivariant cohomology. See
Section \ref{exm} for the details and more examples.

In Section \ref{mainsection}, we will discuss the freeness and the torsion-freeness of $\ZZ[K]$ over $\ZZ[\sfR^*]$. In
particular, we show that the equivariant cohomology of a toric orbifold is torsion-free over $\ZZ[\sfR^*]$. We also give
a certain injectivity theorem of an equivariant cohomology of $\calZ_K$ (Theorem \ref{Injectivity}) for a symplectic
toric orbifold $[\calZ_K/\sfG]$, which give a sufficient condition that $\ZZ[K]$ is free over a subring of
$\ZZ[\sfR^*]$.

Finally, in section \ref{secgysin}, in light of Theorem \ref{Franz}, we construct an algebraic Gysin sequence for
$\Tor$ of $\ZZ[K]$ in analogy with the Gysin sequence of $S^1$-fibration over a toric manifold.
\section{{\bf Moment Angle Complexes and Toric Orbifolds}}\label{bg}
In this section, we review the basic facts about the moment angle complexes and various toric orbifolds to motivate our
results.
\subsection{Moment Angle Complexes}
The \emph{moment angle complex} $\calZ_K$ associated to a simplicial complex $K$ was introduced by Buchstaber and Panov
in \cite{BP99} as a disc-circle decomposition of the Davis-Januszkiewicz universal space associated to a simplicial
complex $K$ \cite{DavisJanuszkiewicz91} and it has been actively studied in \emph{toric topology} and its connections to
symplectic and algebraic geometry and combinatorics. For convenience, we use the following notation for the rest of the
paper.
\begin{notation}\label{notation0}
Let $X,Y$ be the subsets of a set $Z$. For a subset $\sigma \subset [m]$, $X^{\sigma} \times Y^{[m]\backslash\sigma}
\subset Z^m$ is the direct product of $X$ and $Y$'s where $i$-th compont is $X$ if $i \in\sigma$ and $Y$ if
$i\not\in\sigma$.
\end{notation}
For a simplicial complex $K$ on vertices $[m]:=\{1,\cdots,m\}$, the moment angle complex $\calZ_K \subset \CC^m$ is
defined as $\calZ_K = \bigcup_{\sigma \in K} \sfD^{\sigma} \times (\partial \sfD)^{[m]\backslash \sigma}$ where $\sfD=\{
z \in \CC \ | \ |z| \leq 1\}$ is the unit disk and $\partial \sfD$ is its boundary circle. This space $\calZ_K$ carries
a natural action of $\sfT=\U(1)^m$. It is originally proved in \cite{DavisJanuszkiewicz91} that
\begin{eqnarray}\label{Z[T]}
H_{\sfT}^*(\calZ_K,\ZZ) \cong \ZZ[K] \ \ \ \ \mbox{ as graded rings over $\ZZ[\sfT^*]$}.
\end{eqnarray} 
Here $\ZZ[K]$ is the \emph{Stanley-Reisner (face) ring} defined by $\ZZ[K] = \frac{\ZZ[x_1,\cdots, x_m]}{\lan
  x_{\sigma}, \sigma \not\in K\ran}$ where $x_{\sigma}:=\prod_{i\in\sigma} x_i$. With the identification
$\ZZ[\sfT^*]:=H^*(B\sfT,\ZZ) = \ZZ[x_1,\cdots, x_m]$, the isomorphism is as graded algebras over the polynomial ring
with $\deg x_i = 2$. For the details, we refer to Chapter 6 \cite{BP}.

Baskakov-Buchstaber-Panov \cite{BBP04} also computed the ordinary cohomology of $\calZ_K$:
\begin{eqnarray}\label{Z}
H^*(\calZ_K,\ZZ) \cong \Tor_{\ZZ[\sfT^*]}^*(\ZZ[K],\ZZ) \ \ \ \ \mbox{ as graded rings}.
\end{eqnarray} 
Here the grading on the right hand side is the total degree of bidegree coming from the (co)homological degree of Koszul complex and the degree of $\ZZ[K]$. More precisely
\begin{defn}\label{cohdegree}
Let $\sfM$ be a graded $\ZZ[\sfR^*]:=\ZZ[u_1,\cdots,u_n]$-module. Let $\Lambda$ be the exterior algebra generated by
$\eta_1,\cdots,\eta_n$ with $\deg \eta_i=1$. Let $\ZZ[\sfR^*]\otimes^{\sfR}\Lambda$ be the Koszul complex. Then
$\Tor^{\ZZ[\sfR^*]}_*(\sfM,\ZZ)$ is the homology of the complex $\sfM\otimes_{\ZZ[\sfR^*]}
\ZZ[\sfR^*]\otimes^{\sfR}\Lambda$ where the degree is given by $\deg \eta_i=1$. The complex $\sfM\otimes_{\ZZ[\sfR^*]}
\ZZ[\sfR^*]\otimes^{\sfR}\Lambda$ is also a bigraded differential complex with $\bideg (\alpha \otimes \xi_i) = [-1,2] +
   [0,|\alpha|]$ where $\alpha \in M^{|\alpha|}$. The cohomological degree of $\Tor$ is defined to be the total degree
   of this bidegree and is denoted by the superscript as in $\Tor_{\ZZ[\sfR^*]}^*(\sfM,\ZZ)$.
\end{defn}
Now it is natural to ask if $H_{\sfG}^*(\calZ_K,\ZZ)$ can be computed by $\Tor_{\ZZ[\sfR^*]}^*(\calZ_K,\ZZ)$ where $\sfG
\subset \sfT$ is a (possible disconnected) subgroup, $\sfR:=\sfT/\sfG$ and $\ZZ[\sfR^*]:=H^*(B\sfR) \subset
\ZZ[\sfT^*]$. In Section \ref{TheoremFranz}, we show that
\[
H_{\sfG}^*(\calZ_K,\ZZ) \cong \Tor_{\ZZ[\sfR^*]}^*(\calZ_K,\ZZ) \ \  \mbox{ as graded $\ZZ[\sfT^*]$-modules (Theorem \ref{Franz})}.
\]
\subsection{Partial Quotient of Moment Angle Complexes}
When a subgroup $\sfG \subset \sfT$ acts on $\calZ_K$ locally freely, the \emph{quotient stack} $[\calZ_K/\sfG]$ is a
topological orbifold (as a stack), together with the residual action of $\sfR:=\sfT/\sfG$. Indeed, $[\calZ_K/\sfG]$ is
topologically isomorphic to various toric ``spaces", including \emph{quasi-toric orbifolds} defined by Poddar-Sarkar
\cite{PoddarSarkar10}, symplectic compact toric orbifolds defined by Lerman-Tolman \cite{LT}, and \emph{algebraic toric
  orbifolds} defined by Borisov-Chen-Smith \cite{BCS}. In the next section, we recall the construction of those spaces
and see that the cohomology rings of all of these toric spaces in nice cases are computed as the quotient of the
Stanley-Reisner ring $\ZZ[K]$.

In this section, we give a criteria for the local freeness of the action of a subgroup $\sfG$ of $\sfT$ on $\calZ_K$
and a remark about the cohomology of orbifolds as stacks.
\begin{lem}\label{lemlf}
Let $n$ be the largest cardinality of a face in $K$. If a subgroup $\sfG \subset \sfT$ acts on $\calZ_K$ locally freely,
then $\dim \sfG \leq m-n$. Furthermore, if $n$ is the cardinality of maximal faces in $K$ (pure) and $\dim \sfG=m-n$,
then $\sfG$ acts on $\calZ_K$ locally freely if and only if $\sfT_{\sigma} := \U(1)^{\sigma} \times \{1\}^{[m]\backslash
  \sigma} \subset \sfT=\U(1)^m$ surjects to $\sfR:=\sfT/\sfG$ for all maximal faces $\sigma$.
\end{lem}
\begin{proof} 
Let $\sigma \in K$ such that $|\sigma|=n$. Let $0_{\sigma} \in \sfD^{\sigma}\times (\partial \sfD)^{[m]\backslash
  \sigma}$ such that $i$-th component of $0_{\sigma}$ for $i\in\sigma$ is $0 \in \sfD$. Then the stabilizer of
$0_{\sigma}$ in $\sfT$ is $T_{\sigma} := \U(1)^{\sigma} \times \{1\}^{[m]\backslash \sigma} \subset
\sfT=\U(1)^m$. Consider the map $\theta:\sfG \times \sfT_{\sigma} \to \sfT, (g,t) \mapsto gt$. The kernel of this map
$\theta$ is finite if and only if the stabilizer of $0_{\sigma}$ in $\sfG$ is finite. Therefore the local freeness of
the $\sfG$-action implies that the dimension of the image of $\theta$ is $\dim \sfG + |\sigma| = \dim \sfG + n$. Thus
$\dim \sfG \leq m-n$ since $\dim \sfT =m$. To prove the latter claim, note that the local freeness of the
$\sfG$-action is equivalent to that the stabilizer of $0_{\sigma}$ in $\sfG$ is finite for each maximal face
$\sigma$. Since $\sfG=m-n$ and $\dim \sfT_{\sigma}$, $\sfG \cap \sfT_{\sigma}$ is zero dimensional if and only if
$\sfT_{\sigma} \to \sfR$ is surjective.
\end{proof}
\begin{rem}\label{stackcohomology}
In \cite{Edidin10}, Edidin defined the integral cohomology of a stack and showed that if the stack is given as a global
quotient stack $[M/\sfG]$, then $H^*([M/\sfG],\ZZ)$ is canonically isomorphic to $H^*_{\sfG}(M,\ZZ)$. If $\sfG$ acts
locally freely on $M$, the quotient stack $[M/\sfG]$ is an orbifold. The cohomology of an orbifold $[M/\sfG]$ as a stack
is then $H_{\sfG}(M,\ZZ)$. On the other hand, the projection map $B\sfG\times_{\sfG} M \to M/\sfG$ where $M/\sfG$ is the
quotient topological space induces an isomorphism $H_{\sfG}^*(M,\QQ) \cong H^*(M/\sfG,\QQ)$ since the fiber is ``$\QQ$-acyclic".
If $\sfG$ acts freely on $M$, then $H^*([M/\sfG],\ZZ) \cong H_{\sfG}^*(M,\ZZ) \cong H^*(M/\sfG,\ZZ)$.

If $\sfL$ acts on $M$ and $\sfG$ is a subgroup of $\sfL$ that acts on $M$ locally freely, we have the action of
$\sfK:=\sfL/\sfG$ on the orbifold $[M/\sfG]$. In this case, there is an isomorphism of stacks $[[M/\sfG]/\sfK] \cong
[M/\sfL]$ and we can define $H_{\sfK}^*([M/\sfG], \ZZ) := H^*([M/\sfL],\ZZ) = H^*_{\sfL}(M,\ZZ)$ (c.f. \cite{Rom,
  LerMal}).
\end{rem}
\begin{rem}
The following is a useful criteria for the connectedess of $\sfG$. Let $B$ be the integer matrix induced from the quotient map $\sfT \to \sfR$. Then $\sfG$ is connected if and only if $B: \ZZ^m\to \ZZ^n$ is surjective.
\end{rem}
\subsection{Quasi-toric Orbfiolds}\label{PoddarSarkar}
Quasi-toric manifolds are introduced and studied by Davis-Januszkiewicz \cite{DavisJanuszkiewicz91} and its orbifold
counterpart is studied by Poddar-Sarkar \cite{PoddarSarkar10}. Let $\Delta$ be a simple polytope of dimensional $n$ in
$\sfR^n$. Let $H_1,\cdots, H_m$ be the facets of $\Delta$ and for a face $F_{\sigma} = \cup_{i\in\sigma} H_i$, let
$\sfT_{\sigma} := \U(1)^{\sigma} \times \{1\}^{[m]\backslash \sigma} \subset \sfT=\U(1)^m$. Define $\calZ_{\Delta} :=
\sfT \times \Delta/\!\!\sim$ where $(t_1,p)\sim (t_2,q)$ if and only if $p=q$ is contained in a relative interior of
$F_{\sigma}$ and $t_1t_2^{-1} \in \sfT_{\sigma}$ (c.f. Definition 6.1 \cite{BP}). It is known that $\calZ_{\Delta}$ is a
smooth manifold (c.f. Lemma 6.1 \cite{BP}). Let $B$ be an integer $n\times m$ matrix such that for each vertex
$H_{i_1}\cap \cdots \cap H_{i_n}$ of $\Delta$, the corresponding columns $\lambda_{i_1},\cdots, \lambda_{i_n}$ form a
basis of $\QQ^n$. By the assumption, $B$ defines a surjective map $\sfT \surj \sfR$ also denoted by $B$. Let $\sfG$ be
the kernel of $B$. A \emph{quasi-toric orbifold} for the pair $(\Delta, B)$ is defined as the quotient stack
$[\calZ_{\Delta}/\sfG]$. Here note that the assumption on $B$ is equivalent to the local freeness of the $\sfG$-action
on $\calZ_{\Delta}$ (See Lemma \ref{lemlf}).  Since $\calZ_{\Delta}$ is $\sfT$-equivariantly homeomorphic to
$\calZ_{K_{\Delta}}$ where $K_{\Delta}$ is the simplicial complex associated to $\Delta$ (see Section 6.2. \cite{BP}),
the quasi-toric orbifold $[\calZ_{\Delta}/\sfG]$ is topologically the quotient of the moment angle complex
$\calZ_{K_{\Delta}}$ by $\sfG$. As a consequence and by Remark \ref{stackcohomology}, the $\sfR$-equivariant cohomology
of the quasi-toric orbifold $[\calZ_{\Delta}/\sfG]$ is
\[
H_{\sfR}^*([\calZ_{\Delta}/\sfG],\ZZ) \cong H_{\sfT}^*(\calZ_{\Delta},\ZZ) \cong H^*_{\sfT}(\calZ_{K_{\Delta}}, \ZZ) \cong \ZZ[K_{\Delta}].
\]
The rational cohomology ring of the quasi-toric orbifold is computed by Poddar-Sarkar \cite{PoddarSarkar10}:
\begin{thm}[Poddar-Sarkar]\label{qtQ}
If $[\calZ_{\Delta}/\sfG]$ is a quasi-toric orbifold, $H^*([\calZ_{\Delta}/\sfG], \QQ) \cong {\QQ[K_{\Delta}]}/{\lan u_1,\cdots, u_n\ran}$.
\end{thm}
Here $u_j=\sum_{i=1}^mB_{ji}x_i \in \ZZ[x_1,\cdots, x_m]$ and we identify $H^*(B\sfR;\ZZ)=\ZZ[u_1,\cdots, u_n]$. A
quasi-toric orbifold given by $(\Delta,B)$ is a \emph{quasi-toric manifold} if and only if for each vertex $H_{i_1}\cap
\cdots \cap H_{i_n}$ of $\Delta$, the corresponding columns $\lambda_{i_1},\cdots, \lambda_{i_n}$ of $B$ form a
$\ZZ$-basis. In this case, the isomorphism holds with $\ZZ$-coefficients:
\begin{thm}[Davis-Januszkiewicz \cite{DavisJanuszkiewicz91}]\label{DJsmooth}
If $[\calZ_{\Delta}/\sfG]$ is a quasi-toric manifold, $H^*([\calZ_{\Delta}/\sfG], \ZZ)$ is isomorphic to
$\ZZ[K_{\Delta}]/\lan u_1,\cdots, u_n\ran$. Moreover $H^*([\calZ_{\Delta}/\sfG], \ZZ)$ has no $\ZZ$-torsion (which follows from
\cite[Theorem 3.1]{DavisJanuszkiewicz91} and the fact that a quasi-toric manifold is closed and orientable).
\end{thm}
Since $\calZ_{\Delta}$ is $\sfT$-equivariantly homeomorphic to $\calZ_{K_{\Delta}}$, we have
\begin{cor}\label{quasiDM}
If $K_{\Delta}$ and $B$ give a quasi-toric orbifold, then $H_{\sfG}^*(\calZ_{K_{\Delta}};\QQ) \cong \QQ[K_{\Delta}]/\lan u_1,\cdots, u_n\ran$. If $K_{\Delta}$ and $B$ give a quasi-toric manifold, then $H_{\sfG}^*(\calZ_{K_{\Delta}};\ZZ) \cong \ZZ[K_{\Delta}]/\lan u_1,\cdots, u_n\ran$.
\end{cor}
\subsection{Compact Symplectic Toric Orbifolds} \label{LermanTolman}
Lerman-Tolman \cite{LT} classified compact symplectic toric (effective) orbifolds in terms of \emph{labeled
  polytopes}. A labeled polytope $(\Delta, \sfb)$ is a rational simple polytope $\Delta$ in $\RR^n$ with each facet
$H_i, i=1,\cdots,m$ is labeled by a positive integer $\sfb_i$. If $\beta_i$ is the integral primitive inward normal
vector to each facet $H_i$, then by assigning the integer matrix $B=[\sfb_1\beta_1,\cdots, \sfb_m\beta_m]$, we obtain a
quasi-toric orbifold given by $(\Delta, B)$. Here the symplectic structure on $[\calZ_{\Delta}/\sfG]$ comes from
identifying $\calZ_{\Delta}$ with the level set for the reduction of $\CC^m$ by the action of $\sfG$. A \emph{compact
  symplectic toric manifolds} is given by the labeled polytope such that $\sfb_i=1, \forall i=1,\cdots,m$ and such that
for for each vertex $H_{i_1}\cap \cdots \cap H_{i_n}$ of $\Delta$, the corresponding primitive normal vectors
$\beta_{i_1},\cdots, \beta_{i_n}$ of $B$ form a $\ZZ$-basis. This is exactly the Delzant condition in the classification
of compact symplectic manifolds.
\subsection{Algebraic Toric (effective) orbifold (a.k.a. toric Deligne-Mumford stack)}\label{XK}
Let $K$ be a pure simplicial complex on $[m]$. Define a fan $\Sigma_K$ in $\RR^m$ where each cone is generated by the
part of standard basis $\sfe_i, i\in \sigma$ for each $\sigma \in K$. The corresponding toric variety $X_{\Sigma_K}$ is
a smooth open subvariety of $\CC^m$ that is exactly the complement of subspace arrangements given by $K$ (c.f. Chapter 8
\cite{BP}). There is a natural embedding of $\calZ_K$ into $X_{\Sigma_K}$ and it is shown that
\begin{prop}[\cite{BP} Proposition 8.9]\label{propret}
There is a $\sfT$-equivariant deformation retract for $\calZ_K \subset X_{\Sigma_K}$ and, in particular, $H_{\sfG}^*(\calZ_K;\ZZ) \cong H_{\sfG_{\CC}}^*(X_{\Sigma_K};\ZZ)$.
\end{prop}
The algebraic toric orbifolds studied by Borisov-Chen-Smith \cite{BCS} are defined by the \emph{stacky fan}. There they
consider possibly noneffecive orbifolds. In this paper, since we are interested in the effective case, we simplify the
stacky fan and call it the \emph{labeled fan}. A \emph{labeled fan} $(\Sigma, \sfb)$ is a simplicial fan in $\RR^n$ with
each ray $\rho_i$ is labeled by a positive integer $\sfb_i$ where $i=1,\cdots,m$. Let $K$ be the simplicial complex
associated to $\Sigma$. Let $\beta_i$ be the integral primitive generator of each ray $\rho_i$, define an integral
$n\times m$ matrix $B:=[\sfb_1\beta_1,\cdots, \sfb_m\beta_m]$, and let $\sfG$ be the kernel of the induced map of tori
$B:\sfT \to \sfR$. The \emph{toric Deligne-Mumford(DM) stack} associated to a labeled fan $(\Sigma, \sfb)$ is defined as
the quotient stack $\calX_{\Sigma,\sfb}:=[X_{\Sigma_K}/\sfG_{\CC}]$ where $\sfG_{\CC}$ is the complexification of
$\sfG$.

A toric DM stack $\calX_{\Sigma,\sfb}$ (or its labeled fan $(\Sigma,\sfb)$) is \emph{complete} if the fan is complete
i.e. the union of cones is $\RR^n$. A toric DM stack $\calX_{\Sigma,\sfb}$ is a non-singular toric variety if and only
if the labels $\sfb_i=1$ and the fan $\Sigma$ is non-singular, i.e. for each maximal cone, the primitive generators of
the rays in the cone, $\beta_{i_1},\cdots,\beta_{i_n}$, form a $\ZZ$-basis. In this case, we call $(\Sigma,\sfb)$
non-singular. For a non-singular and complete labeled fan, we have the following classical result:

\begin{thm}[Danilov\cite{Danilov78}, Jurkiewicz\cite{Jurkiewicz85}]
If $(\Sigma,\sfb)$ is non-singular and complete, then $H^*([\calX_{K_\Sigma}/\sfG_{\CC}],\ZZ) \cong \ZZ[K_{\Sigma}]/\lan
u_1,\cdots, u_n\ran$. Furthermore, it has no $\ZZ$-torsion (\cite{Danilov78} Theorem 10.8).
\end{thm}
In Proposition 3.7 \cite{BCS}, it is proved that the \emph{coarse moduli space} (underlying algebraic variety) for
$[X_{K_{\Sigma}}/\sfG_{\CC}]$ is exactly the toric variety $X_{\Sigma}$ (See also \cite{Cox95}) \footnote{Even if the
  toric orbifold is a non-trivial orbifold, its coarse moduli space could be a non-singular toric variety. The simplest
  example of such a case may be the weighted projective space $[\CP^1_{12}]=[\CC^2\backslash \{(0,0)\} /\CC^{\times}]$
  where the action of $\CC^{\times}$ is weighted by $(1,2)$. Its coarse moduli space is simply $\CP^1$.}. Thus for more
general cases, one still has an isomorphism with $\QQ$-coefficients:
\begin{thm}[Danilov\cite{Danilov78}]\label{dmsQ}
If $(\Sigma,\sfb)$ is complete, then $H^*([X_{\Sigma_K}/\sfG_{\CC}],\QQ) \cong \QQ[K]/\lan u_1,\cdots, u_n\ran$ where
$u_j=\sum_{i=1}^mB_{ji}x_i$.
\end{thm}
Theorem \ref{propret} and Remark \ref{stackcohomology} therefore imply
\begin{cor}\label{algDM}
If $K$ and $B$ are given by a complete labeled fan, then $H_{\sfG}^*(\calZ_K;\QQ) \cong \QQ[K]/\lan u_1,\cdots, u_n\ran$. If $K$ and $B$ are given by a complete and non-singular labeled fan, then $H_{\sfG}^*(\calZ_K;\ZZ) \cong \ZZ[K]/\lan u_1,\cdots, u_n\ran$.
\end{cor}

Note that the underlying combinatorial structures for quasi-toric orbifolds and toric DM stacks are both simplicial
complexes. Any symplectic toric orbifold can be made into an algebraic one by taking the normal fans to the polytopes.
However, not all quasi-toric orbifolds can be made algebraic. A toric DM stack associated to a \emph{polytopal fan} can be
made into a symplectic toric orbifold but there is a toric DM stack associated to a \emph{non-polytopal} fan. Such a
toric DM stack can not be realized even as a quasi-toric orbifold.

In the light of Theorem \ref{Franz} and Corollary \ref{quasiDM} and \ref{algDM}, it is natural to study the following question
\begin{question}
When is $H_{\sfG}^*(\calZ_K;\ZZ)$ a quotient of the Stanley-Reisner ring for a general subgroup $\sfG$?
\end{question}
Our answer to this question is Theorem \ref{main}. Also we will see in Section \ref{mainsection} that when the $\sfG$-action is locally free and $H_{\sfG}^*(\calZ_K;\QQ)$ is a quotient of Stanley-Reisner ring, then the dimension of $\sfG$ must be maximal.
 
\section{{\bf The Proof that $H_{\sfG}^*(\calZ_K;\ZZ) \cong \Tor^*_{\ZZ[\sfR^*]}(\ZZ[K],\ZZ)$}}\label{TheoremFranz}
In this section, we prove that $H_{\sfG}^*(\calZ_K;\ZZ)$ is isomorphic to $\Tor^*_{\ZZ[\sfR^*]}(\ZZ[K],\ZZ)$ as a graded module over $\ZZ[\sfT^*]$. The idea of the proof, especially to use the homological machinery developed in \cite{Franz02}, was communicated to us by Franz. Throughout, we will use terminology found in \cite{Franz02}.

We will use the following notation consistently throughout this paper unless otherwise specified:
\begin{notation}\label{notation}
Let $K$ be a simplicial complex on $[m]:=\{1,\cdots,m\}$ (possibly with ghost vertices) and let $\calZ_K$ be the
associated moment angle complex with the standard torus $\sfT:=\U(1)^m$-action. Let $\frakt:=\Lie(\sfT) =\RR^m$ and let
$\sfN_{\sfT} = \ZZ^m$ be the kernel of the exponential map $\frakt \to \sfT$. Let $\sfG \subset \sfT$ be a
(\emph{possibly disconnected}) subgroup of dimension $m-n$ and let $\sfR:=\sfT/\sfG$ be the quotient torus. We identify
$\sfR \cong \U(1)^n$ so that the quotient map $\sfT \to \sfR$ defines an integral $n\times m$ matrix $B$ which is viewed
as the surjective linear map $\frakt \to \frakr:=\Lie \sfR$. 

Let $\ZZ[\sfT^*]:=H^*(B\sfT;\ZZ) = \ZZ[x_1,\cdots, x_m]$ where $\{x_j\}$ is the standard basis of $\sfN_{\sfT}^*$ and
let $\ZZ[\sfR^*]:=H^*(B\sfR;\ZZ) = \ZZ[u_1,\cdots, u_n]$ where $\{u_i\}$ is the standard basis of $\sfN_{\sfR}^*$. We
regard $\ZZ[\sfR^*]$ as a subring of $\ZZ[\sfR^*]$ so that $u_i := \sum_{j=1}^m B_{ij}x_j$. The Stanley-Reisner ring
$\ZZ[K]$ is defined as the quotient of $\ZZ[\sfT^*]$ by the monomials corresponding to non-faces of $K$.
\end{notation}

\begin{defn}
$\sfM$ is an \emph{$H^*(B\sfR)$-module up to homotopy} if it is a module over the reduced cobar construction of $H^*(\sfR)$ (Section 4 \cite{Franz05i}). 
\end{defn}
\begin{thm}\label{Franz}
Under Notation \ref{notation},  there is an isomorphism of graded modules over $\ZZ[\sfT^*]$.
\[
\Theta_{\sfR}: H^*_{\sfG}(\calZ_K,\ZZ) \to \Tor^*_{\ZZ[\sfR^*]}(\ZZ[K],\ZZ)
\]
where the cohomological grading on the right hand side is given in Definition \ref{cohdegree}. 
\end{thm}
\begin{proof}
We suppress the coefficient ring $\ZZ$. Given a map $p: Y \to B\sfR$, the twisted tensor product $C^*(Y)\otimes^{\sfR}
H^*(\sfR)$ is defined by (5.7) \cite{Franz02}. Proposition 5.2 in \cite{Franz02} states that there is a
quasi-isomorphism of differential graded (dg) $H_*(\sfR)$-modules $\Phi_Y^*: C^*(Y)\otimes^{\sfR} H^*(\sfR) \to C^*(Y\times^{B\sfR} E\sfR)$
where $Y\times^{B\sfR} E\sfR$ is a pullback of $E\sfR \to B\sfR$ along $p:Y \to B\sfR$.

Let $Y:=E\sfR\times_{\sfR}\left(E\sfT\times_{\sfG} \calZ_K\right)$ and $p: Y \to B\sfR$  is defined as the composition of maps
\[
\xymatrix{
E\sfR\times_{\sfR}\left(E\sfT\times_{\sfG} \calZ_K\right) \ar[r]_{\ \ \ \ \ \ \ \sfq_{\sfR}} &E\sfT \times_{\sfT} \calZ_K \ar[r]& B\sfT \to B\sfR,
}
\]
where $\sfq_{\sfR}$ is the projection to the 2nd and 3rd components, the second map is the projection to the 1st
component, and the last map is a classifying map for $B: \sfT \to \sfR$. We observe that $p$ is obtained by taking the
quotient of
\[
E\sfR \times \left(E\sfT\times_{\sfG} \calZ_K\right) \to E\sfT\times_{\sfG}\calZ_K \to E\sfT/\sfG \to E\sfR
\]
by the free actions of $\sfR$ on each space. This implies that $Y\times^{B\sfR} E\sfR = E\sfR \times \left(E\sfT\times_{\sfG} \calZ_K\right)$.

Now Proposition 5.2 \cite{Franz02} states that we have the quasi-isomorphism:
\[
\Phi^*_Y: C^*(E\sfR\times_{\sfR}\left(E\sfT\times_{\sfG} \calZ_K\right)) \otimes^{\sfR} H^*(\sfR) \to C^*(E\sfR\times \left(E\sfT\times_{\sfG} \calZ_K\right)).
\]
The homology of the right hand side is $H_{\sfG}^*(\calZ_K)$. On the left hand side, since $\sfR$ acts on
$E\sfT\times_{\sfG} \calZ_K$ freely, the fibers of $\sfq_{\sfR}$ are $E\sfR$ and therefore it induces a
quasi-isomorphism of $H^*(B\sfT)$-modules up to homotopy
\[
\sfq_{\sfR}^*: C^*(E\sfT\times_{\sfT}\calZ_K) \to C^*(E\sfR\times_{\sfR}\left(E\sfT\times_{\sfG} \calZ_K\right)),
\]
i.e. it is a homomorphism of dg $C^*(B\sfR)$-modules such that after taking homology, it becomes an isomorphism of
$H^*(B\sfR)$-modules. Theorem 1.1 \cite{Franz05i} implies that $C^*(E\sfT\times_{\sfT} \calZ_K)$ is formal as a
$H^*(B\sfT)$-module up to homotopy, i.e there is a sequence of quasi-isomorphisms connecting $C^*(E\sfT\times_{\sfT}
\calZ_K)$ to $H^*(E\sfT\times_{\sfT} \calZ_K)$ as dg modules over reduced cobar construction of $H^*(\sfT)$, and
therefore as dg modules over reduced cobar construction of $H^*(\sfR)$. Since the operation to take the twisted tensor
product $\otimes^{\sfR} H^*(B\sfR)$ and the homology of it preserves quasi-isomorphisms of $H^*(B\sfR)$-modules up to
homotopy (c.f. Theorem 8.20, \cite{McCleary}), the map $\sfq_{\sfR}^* $ induces a quasi-isomorphism
\[
\widetilde{\sfq_{\sfR}^*}: C^*(E\sfR\times_{\sfR}\left(E\sfT\times_{\sfG} \calZ_K\right)) \otimes^{\sfR} H^*(\sfR)  \to C^*(E\sfT\times_{\sfT}\calZ_K) \otimes^{\sfR} H^*(\sfR).
\]
and there is a sequence of quasi-isomorphisms connecting $C^*(E\sfT\times_{\sfT}\calZ_K) \otimes^{\sfR} H^*(\sfR)$ and
$H^*(E\sfT\times_{\sfT} \calZ_K) \otimes^{\sfR} H^*(\sfR)$. Since the homology of the complex $H^*(E\sfT\times_{\sfT}
\calZ_K) \otimes^{\sfR} H^*(\sfR)$ is $\Tor^*_{\ZZ[\sfR^*]}(\ZZ[K], \ZZ)$, we obtain the isomorphism $\Theta_{\sfR}$. We
summarize all in the following diagram:
\begin{small}
\[
\xymatrix{
H_{\sfG}^*(\calZ_K)  \ar[ddd]^{\cong}_{\Theta_{\sfR}}        &  C^*(E\sfR\times \left(E\sfT\times_{\sfG} \calZ_K\right))  \ar@{=>}[l]_{homology\ \ \ \ \ \ \ }          \\
& C^*(E\sfR\times_{\sfR}\left(E\sfT\times_{\sfG} \calZ_K\right)) \otimes^{\sfR}  H^*(\sfR)\ar[u]_{\Phi^*_{Y}} \ar[d]^{\widetilde{\sfq_{\sfR}^*}} \\
& C^*(E\sfT\times_{\sfT}\calZ_K) \otimes^{\sfR} H^*(\sfR) \ar@{<.>}[d]_{seq\ o\!f\ quasi-iso}\\
\Tor^*_{\ZZ[\sfR^*]}(\ZZ[K], \ZZ)               &  H^*(E\sfT\times_{\sfT} \calZ_K) \otimes^{\sfR} H^*(\sfR) \ar@{=>}[l]^{homology\ \ \ }.          \\
}
\]
\end{small}
The right vertical maps gives a sequence of quasi-isomorphisms of dg $H_*(\sfR)$-modules and at the both ends, we have the desired graded $\ZZ$-modules after taking homology.

To show that the map $\Theta_{\sfR}$ is a homomorphism of modules over $\ZZ[\sfT^*]$, it is sufficient to prove that the following diagram is commutative
\[
\xymatrix{
H^*_{\sfT}(\calZ_K,\ZZ) \ar[r]_{\Theta} \ar[d]_{\iota_{\sfR}^*}& \ZZ[K] \ar[d]_{\phi_{\sfR}}\\
H^*_{\sfG}(\calZ_K,\ZZ) \ar[r]_{\Theta_{\sfR}} &  \Tor^*_{\ZZ[\sfR^*]}(\ZZ[K],\ZZ) \\
}
\]
where $\phi_{\sfR}$ is the obvious map induced from the inclusion of Koszul complexes, $\iota_{\sfR}^*$ is the pullback
of the quotient map $\iota_{\sfR}: E\sfT \times_{\sfG} \calZ_K \to E\sfT\times_{\sfT} \calZ_K$ and $\Theta$ is the
isomorphism mentioned at (\ref{Z[T]}) Section \ref{bg}. Consider the map $\varphi: \sfR \to \one$ where $\one$ is the
trivial group. Since it satisfies the condition in Propostion 4.11 \cite{Franz02}, the naturality stated in Proposition
5.2 \cite{Franz02} implies that the following diagram commutes
\[
\xymatrix{
 C^*(E\sfT\times_{\sfT} \calZ_K) \ar[d]_{\bar\sfq_{\sfR}^*} & C^*(E\sfT\times_{\sfT} \calZ_K) \otimes^{\one} H^*(\one) \ar[l]  \ar[d]_{\bar\phi_{\sfR}}\\
  C^*(E\sfR\times \left(E\sfT\times_{\sfG} \calZ_K\right)) & C^*(E\sfR\times_{\sfR}\left(E\sfT\times_{\sfG} \calZ_K\right)) \otimes^{\sfR} H^*(\sfR) \ar[l] \\
}
\]
where $\bar\sfq_{\sfR}^*$ is the pullback of the projection $\bar\sfq_{\sfR}: E\sfR\times (E\sfT\times_{\sfG} \calZ_K)
\to E\sfT\times_{\sfT} \calZ_K$ and $\bar\phi_{\sfR}$ is the map induced by $\sfq_{\sfR}$ and $\varphi:\sfR \to
\one$. After taking the homology, $\bar\sfq_{\sfR}^*$ and $\bar\phi_{\sfR}$ naturally coincide with $\iota_{\sfR}^*$ and
$\phi_{\sfR}$ respectively.
\end{proof}
\begin{rem}\label{formality}
Let $X$ be any topological space with the $\sfT$-action such that $C^*(E\sfT\times_{\sfT}X;\ZZ)$ is formal as
$H^*(B\sfT;\ZZ)$-modules up to homotopy. Then Theorem \ref{Franz} also holds for $X$. Namely, the above proof can be
identically applied to this case and gives the isomorphism $\Theta_{\sfR}: H_{\sfG}^*(X;\ZZ) \cong
\Tor_{\ZZ[\sfR^*]}^*(H_{\sfT}^*(X;\ZZ),\ZZ)$.
\end{rem}
In terms of toric orbifolds discussed in Section \ref{bg}, we have 
\begin{cor}\label{cormain}
Let $\calX$ is a quasi-toric orbifold or effective toric Deligne-Mumford stack with associated $K$, $\sfR$ and $\sfT$ in
the sense of Section \ref{bg}. There is an isomorphism of graded modules over $\ZZ[\sfT^*]$
\[
H^*(\calX;\ZZ) \cong \Tor_{\ZZ[\sfR^*]}^*(\ZZ[K],\ZZ).
\]
\end{cor}
\begin{rem} Theorem \ref{Franz} (or Corollary \ref{cormain}) generalizes several results that have been proved. For example:
\begin{itemize}
\item[(1)] In case $\sfG = \sfT$, Theorem \ref{Franz} states that the $\sfT$-equivariant cohomology of $\calZ_K$ is the
  Stanley-Reisner ring of $K$, which is well-known (c.f.\cite{BP}). In case $\sfG = {1}$, then one recovers that the
  ordinary cohomology of $\calZ_K$ is the $\Tor$-algebra of $\ZZ[K]$ over $\ZZ[\sfT^*]$ (Theorem 7.6 \cite{BP}).  One
  may therefore view this result as interpolating between these extreme cases.
\item[(2)] If $\calZ_K/\sfG$ is a quasi-toric manifold, then one recovers Theorem 7.37 \cite{BP}.
\item[(3)] When $X_{\Sigma}:=X_{\Sigma_K}/\sfG_{\CC}$ is the coarse moduli for a toric orbifold, one recovers Theorem 1.2 \cite{Franz05i}:
\[
H^*(X_{\Sigma},\QQ) \cong H^*_{\sfG_{\CC}}(X_{\Sigma_K},\QQ) = H^*_{\sfG}(\calZ_K,\QQ) \cong \Tor^*_{\ZZ[\sfR^*]}(\ZZ[K],\ZZ) \otimes \QQ.
\]
\end{itemize}
\end{rem}
\section{{\bf Basics from commutative algebra}}\label{ba}
In this section, we collect some definitions and basic properties of graded modules over a polynomial ring and discuss the relations among them.
Throughout this section, $R$ will be a polynomial ring in variables $u_1,\dots,u_n$ (generated in degree $2$) over $k=\ZZ$ or $\QQ$, and $M$ will be a graded
$R$-module (though not necessarily finitely generated).  We will denote the ideal of $R$ generated by polynomials of positive degree by $R_+$.

Below we give brief names to several properties of $R$-modules so that we can refer to them later.
\begin{defn}  \label{def:ringConditions} For $M$ and $R$ as above, one says that:
\begin{itemize}
\item[($k$1)] $M$ is \emph{free over $R$} if $M \cong \oplus_{e \in E} R\cdot e$ and $R\cdot e \cong R$ as graded $R$-modules where $E$ is a subset of $M$;
\item[($k$2)] $M$ is \emph{flat over $R$} if $\Tor_{> 0}^R(M,N)=0$ for any (f.g.) module $N$;
\item[($k$3)] $M$ is \emph{torsion-free over $R$} if there is no torsion over $R$ ($x \in M$ is a torsion element over $R$ if $\cdot x : R \to R$ has non-trivial kernel);
\item[($k$4)] $M$ is \emph{a big Cohen-Macaulay $R$-module} if $\Tor_1^R(M,k) = 0$.
\end{itemize}
In general, one has the following implications.
\[
(k1) \Rightarrow (k2) \Rightarrow (k3), (k4)
\]
\end{defn}


We will see that in the case of interest to us, ($k$4) has the usual meaning in terms of regular sequences; see Proposition \ref{p24}.

\begin{defn}
A non-zero element $r \in R$ is \emph{$M$-regular} if $0 \to M \stackrel{r}{\longrightarrow} M$ is exact. A sequence of elements
$f_1,\dots, f_c\subset R$ is an \emph{$M$-regular sequence} if, for each $i\leq c$,  $f_i$ is $(M/(f_1,\dots,f_{i-1})M)$-regular.
\end{defn}
\begin{rem}
We call the condition ($k$4) in Definition \ref{def:ringConditions} \emph{big Cohen-Macaulay} because $\Tor_1^R(M,k) = 0$
is the same as saying that there exists a system of parameters of $R_+$ (namely, the variables of $R$; see Corollary \ref{cor:depthKoszul}) that is an $M$-regular sequence.
The `big' terminology is a reference to the fact that $M$ need not be finitely generated; see \cite[Chapter 8]{BHcomm}
\footnote{In the reference \cite{BHcomm}, it is mentioned that the existence of a big Cohen-Macaulay module over a local ring $R$ is an open problem.
This question is not interesting for $R$, since it is a (non-local) Cohen-Macaulay ring.}.
\end{rem}
\begin{defn}
Let $R_+$ be the ideal generated by the positive degree elements of $R$, and suppose that $M\not=IM$.  One defines the \emph{depth} (as well as \emph{grade}) of $M$ over $R$ by
\begin{equation}\label{depth}
\depth_R(M) := \grade(I,M) := \min\{ i \ |\ \Ext_R^i(R/I, M)\not=0\}.
\end{equation}
If $M = IM$, one sets $\depth_R(M)=\infty$.  When $M$ is finitely generated over $R$, this definition is the usual definition of the depth of $M$
and it is the length of maximal $M$-regular sequence in $R_+$ \cite{BHcomm}.
\end{defn}

Our goal for this section is to show that $\depth_R(M)$ is the length of the longest $M$-regular sequence in $R_+$ when $M$ is only assumed to be
finitely generated over some homomorphic image of $R$, and the $R$ action on $M$ factors through this homomorphic image.  This is recorded in the following
proposition.
\begin{prop}\label{p23}
Let $S = k[x_1,\dots,x_m]$, $M$ be a finitely generated graded $S$-module,
and $\varphi: R \to S$ be a graded ring homomorphism (so that $M$ is hence a graded $R$-module via $\varphi$).
Then all maximal $M$-regular sequences in $R_+$ have the same length $\depth_R(M)$. 
\end{prop}
This proposition is a special case of the following propositon, whose proof will come after some lemmas.
\begin{prop}\label{p24}
Let $\varphi: R \to S$ be a homorphism of Noetherian rings, $M$ a finitely generated $S$-module and $I$ an ideal of $R$ with $IM\not=M$. Then 
\[
\grade(I,M) := \min\{i \ |\ \Ext^i_R(R/I, M) \not=0\}
\]
agrees with the length of all maximal $M$-regular sequences in $I$.
\end{prop}

The fact that allows one to extend the usual finitely generated setup to the generality above is the following lemma.  
\begin{lem}\label{26}
Let $\varphi:R\to S$ be a homomorphism of graded Noetherian rings, $M$ a finitely generated $S$-module, and $N$ a finitely generated $R$-module. Then $\Hom_R(N,M) = 0$ if and only if $\Ann_R(N)$ contains a non-zero divisor on $M$. 
\end{lem}
\begin{rem} \label{rem:assMFinite}
Before starting on the proof, let us remark that in this case, the set of associated primes of $M$ \emph{over $R$} is finite, even though $M$ may not be
finitely generated; see \cite[Exercise 6.9]{MatsumuraComm}.  Indeed, one sees this by taking a primary decomposition of
$M$ over $S$, which is also a primary decomposition over $R$, since the $R$-annihilator of an $S$-module is the preimage of the annihilator in $S$ (and the inverse
image of a primary ideal is primary).
\end{rem}

\begin{proof}[Proof of Lemma \ref{26}]
Suppose that $x \in \Ann_R(N)$ is a non-zero divisor. Then for any $\psi \in \Hom_R(N,M)$, we have
\[
x\psi(n)=\psi(xn)=\psi(0)=0, \forall n\in N.
\]
Since $x$ is a non-zero divisor on $M$, we have $\psi(n)=0$. 

Now assume that $\Ann_R(N)$ consists of zero-divisors on $M$.  As mentioned in Remark \ref{rem:assMFinite},
the set of associated primes of $M$ over $R$ is finite.  Since $\Ann_R(N)$ consists of zerodivisors,
it is contained in the (finite) union of all associated primes of $M$.  Therefore, we can apply the Prime Avoidance Lemma to get $\Ann_R(N) \subset \frakp$
for some associated prime $\frakp$ of $M$ over $R$.  We then have the following non-trivial map:
\[
N_{\frakp} \surj N_{\frakp}/\frakp N_{\frakp} \surj k(\frakp) \inc M_{\frakp},
\]
where $k(\frakp)$ denotes the residue field $R_\frakp/\frakp R_\frakp$.
Thus, since $N$ is finitely generated over $R$, $\Hom_{R}(N,M)_{\frakp}=\Hom_{R_{\frakp}}(N_{\frakp}, M_{\frakp})\not=0$, which gives $\Hom_R(N,M)\not=0$.
\end{proof}
We also record without proof a basic fact from the homological algebra of commutative rings.
\begin{lem}\label{27}
Let $\varphi:R\to S$ be a homomorphism of Noetherian rings, $M$ a finitely generated $S$-module, and $N$ a finitely generated $R$-module. If $(x_1,\cdots, x_r)$ is a regular sequence in $\Ann_R(N)$ for $M$, then 
\begin{equation*}
\Hom_R(N, M/(x_1,\cdots,x_r)M) = \Ext^r_R(N,M).
\xqedhere{5cm}{\qed}
\end{equation*}
\end{lem}
\begin{proof}[Proof of Proposition \ref{p24}] Let $(x_1,\cdots, x_r)$ be a maximal $M$-regular sequence in $I$. By Lemma \ref{27},
\[
\Ext^i_R(R/I, M) \cong \Hom_R(R/I, M/(x_1,\cdots, x_i)M).
\]
If $i < r$, then $x_{i+1}$ is a non-zero-divisor in $M/(x_1,\cdots, x_i)M$, therefore by Lemma \ref{26}, $\Ext^i_R(R/I,
M)=0$. Since $(x_1,\cdots, x_r)$ is maximal, $I$ doesn't contain any non-zero-divisor in $M/(x_1,\cdots, x_r)M$. Thus
$\Ext^r_R(R/I, M)\not=0$. This proves the claim.
\end{proof}
One has the following well known characterization of ($k$4) in terms of $M$-regular sequences due to Serre \cite[Chapter IV.A]{SerreLocal}.
\begin{prop}
Let $u_1,\dots,u_n$ be a homogeneous minimal generating set of $R_+$.  Suppose that $R$ and $M$ satisfy the hypotheses
of Proposition \ref{p23}.  Then the following properties are equivalent:

\begin{enumerate}[a)]
\item $H_p(\mathbf{u},M) = 0$ for $p \geq 1$.
\item $H_1(\mathbf{u},M) = 0$.
\item The sequence $u_1,\dots,u_n$ is $M$-regular.
\end{enumerate}
Here, $H_p(\mathbf{u},M)$ denotes the $p^\text{th}$ Koszul homology of the sequence $u_1,\dots,u_n$ with coefficients in $M$.
\end{prop}

\begin{proof}
The proof that appears in \emph{op. cit.} uses standard techniques of the Koszul complex which hold even for modules which are not
finitely generated over $R$, together with Nakayama's lemma for finitely generated modules over a Noetherian local ring.  Since a version
of Nakayama's lemma holds for graded modules that satisfy our hypothesis, Serre's original argument remains valid.
\end{proof}

\begin{cor}\label{p22}
Let $u_1,\dots, u_n $ be a homogeneous minimal generating set of $R_+$, and suppose that $R$ and $M$ satisfy the hypothesis of Proposition \ref{p24}.
Then $\Tor_1^R(M,k) = 0$ if and only if $(u_1,\cdots, u_n)$ is a regular sequence for $M$.
\end{cor} 
\begin{proof}
The Koszul complex on $u_1,\dots,u_n$ resolves $k$ over $R$, and hence one can use its homology to compute the $\Tor$ modules.  Now appeal to the previous proposition.
\end{proof}
Propositions \ref{p22} and \ref{p23} show that in the setup of \ref{p23}, big Cohen-Macaulayness of $M$ has the usual meaning in terms of the
maximal length of an $M$-regular sequence.
\begin{cor} \label{cor:depthKoszul}
In the setup of Proposition \ref{p23}, one has that $\Tor_1^R(M,k)=0$ if and only if $\depth_R(M) = n$.
\end{cor}

\section{{\bf Properties of $\ZZ[K]$ as an algebra over $\ZZ[\sfR^*]$}}\label{mainsection}
In this section, we start with the characterization of the big Cohen-Macaulayness ($k$4), and then discuss the freeness
($k$1) and the torsion-freeness ($k$3), for $\ZZ[K]$ as a ring over $\ZZ[\sfR^*]$. In the rest of the paper, we use
Notation \ref{notation} unless otherwise specified.
\subsection{Big Cohen-Macaulayness}
The following theorem is a variant of Theorem 1.1 \cite{FranzPuppe07} and Lemma 5.1 in \cite{Franz09}. The differences from them are that the $\sfT$-CW complex $E\sfG\times_{\sfG} \calZ_K$ is not finite and that we consider the cohomology of the quotient stack $[\calZ_K/\sfG]$ instead of the one of the underlying topological space $\calZ_K/\sfG$. 
\begin{thm}\label{main} Let $\iota_{\sfR}: E\sfT\times_{\sfG}\calZ_K \to E\sfT\times_{\sfT} \calZ_K$ be the quotient map by the $\sfR$-action. The following are equivalent:
\begin{itemize}\label{torsur}
\item[(1)] $\ZZ[K]$ is big Cohen-Macaulay over $\ZZ[\sfR^*]$, i.e. $\Tor^{\ZZ[\sfR^*]}_1(\ZZ[K],\ZZ)=0$.
\item[(2)] $\iota^*_{\sfR}:H_{\sfT}^*(\calZ_K;\ZZ) \to H_{\sfG}^*(\calZ_K;\ZZ)$ is surjective.
\item[(3)] $H_{\sfG}^{odd}(\calZ_K;\ZZ)=0$.
\end{itemize}
\end{thm}
\begin{proof}
Since $\Tor^{\ZZ[\sfR^*]}_{0}(\ZZ[K],\ZZ) =\ZZ[K]/\lan u_1,\cdots, u_n\ran$, Theorem \ref{Franz} implies that (2) is
equivalent to the vanishing of $\Tor^{\ZZ[\sfR^*]}_{i>0}(\ZZ[K],\ZZ)$, which is actually equivalent to (1) by
Proposition 2.3 \cite{FranzPuppe07}.  (2) implies (3) because $H_{\sfT}^*(\calZ_K;\ZZ)$ has only even degree
classes. Now (3) implies that the Serre spectral sequence for the Borel construction for the residual $\sfR$-action for
$E\sfT\times_{\sfG}\calZ_K$ degenerates at $E_2$ level and hence the pullback of the fiber inclusion
$E\sfT\times_{\sfG}\calZ_K \inc E\sfR\times_{\sfR} (E\sfT\times_{\sfG}\calZ_K)$ is surjective. This pullback can be
identified as $\iota^*$ and thus (3) implies (2) (See also Lemma 5.1 \cite{Franz09} and its proof).
\end{proof}
Again it is worth noting that Theorem \ref{main} for any $\sfT$-space $X$ that satisfies the formality (see Remark
\ref{formality}) and such that $H_{\sfT}^{odd}(X;\ZZ)$.
\begin{rem}\label{surj=quot}
By Theorem \ref{Franz}, $\iota^*_{\sfR}$ is surjective if and only if $\Theta: H_{\sfG}^*(\calZ_K,\ZZ) \cong \ZZ[K]/\lan u_1,\cdots, u_n\ran$.
\end{rem}
\begin{rem}\label{mainoverQ}
Theorem \ref{main} holds after replacing $\ZZ$ by $\QQ$. 
\end{rem}
\begin{cor}\label{cormain2}
If $\calX=[\calZ_K/\sfG]$ is a quasi-toric orbifold or effective toric Deligne-Mumford stack, then the following are equivalent
\begin{itemize}
\item[(1)] $\Tor^{\ZZ[\sfR^*]}_1(\ZZ[K],\ZZ)=0$.
\item[(2)] $H^*(\calX;\ZZ)$ is the quotient of $\ZZ[K]$ by $u_1,\cdots, u_n$.
\item[(3)] $H^{odd}(\calX;\ZZ)=0$.
\end{itemize}
\end{cor}
\subsection{Freeness}
The following theorem is analogous to Lemma 6.1 \cite{Franz09}.
\begin{prop}\label{free1} $\ZZ[K]$ is free over $\ZZ[\sfR^*]$ if and only if $H_{\sfG}^*(\calZ_K;\ZZ)$ has no $\ZZ$-torsion and has only even degree.
\end{prop}
\begin{proof}
If $H_{\sfG}(\calZ_K,\ZZ)$ has only even degree, then $\iota^*_{\sfR}: H_{\sfT}^*(\calZ_K,\ZZ) \to
H_{\sfG}^*(\calZ_K,\ZZ)$ is surjective by Theorem \ref{torsur}. The surjectivity implies that $H^r_{\sfG}(\calZ_K,\ZZ)$
has finite rank for each $r$ and is actually a finitely generated free $\ZZ$-module if it has no $\ZZ$-torsion. The
Leray-Hirsch Theorem (c.f.\ Theorem 4D.1 \cite{HatcherAT}) can be applied to the fiber bundle
$E\sfR\times_{\sfR}(E\sfT\times_{\sfG} \calZ_K) \to B\sfR$ where the pullback along the fiber $E\sfT \times_{\sfG}
\calZ_K$ can be identified with $\iota^*$ and therefore we have the isomorphism $\ZZ[\sfR^*]\otimes_{\ZZ}
H^*_{\sfG}(\calZ_K;\ZZ) \cong \ZZ[K]$. Since $H^*_{\sfG}(\calZ_K;\ZZ)$ is a free $\ZZ$-module, $\ZZ[K]$ is a free
$\ZZ[\sfR^*]$-module.
 
 On the other hand, the freeness of $\ZZ[K]$ over $\ZZ[\sfR^*]$ implies $\Tor_1^{\ZZ[\sfR^*]}(\ZZ[K],\ZZ) = 0$ and so
 $\iota^*_{\sfR}$ is surjective by Theorem \ref{torsur}. Thus $H_{\sfG}(\calZ_K,\ZZ) \cong
 H_{\sfT}(\calZ_K,\ZZ)\otimes_{\ZZ[\sfR^*]}\ZZ$. By freeness, we can write $H_{\sfT}(\calZ_K;\ZZ)\cong \bigoplus_{e\in
   E} \ZZ[\sfR^*]e$ as a free $\ZZ[\sfR^*]$-module, where $e$'s are even degree classes. Then
 $H_{\sfT}(\calZ_K;\ZZ)\otimes_{\ZZ[\sfR^*]}\ZZ \cong \oplus_{e\in E}\ZZ e$. Thus there are no $\ZZ$-torsions and no odd
 degree classes.
\end{proof}
The same proof as above proves the following (see also Remark \ref{mainoverQ}):
\begin{prop}\label{Q}
$\QQ[K]$ is free over $\QQ[\sfR^*]$ if and only if $H_{\sfG}^*(\calZ_K;\QQ)$ has no odd degree classes.
\end{prop}
With this proposition, together with the local freeness of the $\sfG$-action, we can also prove the following lemma.
\begin{lem}\label{fgQ}
Suppose that the $\sfG$-action on $\calZ_K$ is locally free. If $H_{\sfG}^*(\calZ_K;\QQ)$ has no odd degree, then $\QQ[K]$ is finitely generated over $\QQ[\sfR^*]$.
\end{lem}
\begin{proof}
Since the $\sfG$-action on the smooth variety $X_{\Sigma_K}$ defined in Section \ref{XK} is locally free, we have the
differentiable orbifold $[X_{\Sigma_K}/\sfG]$. By the construction of de Rham cohomology for differentiable orbifolds,
c.f. $\S$ 2.1 \cite{ALR}, $H^*([X_{\Sigma_K}/\sfG]; \RR)$ is finitely dimensional. Since $\calZ_K \inc X_{\Sigma_K}$ is
a $\sfT$-equivariant deformation retract, $H_{\sfG}^*(\calZ_K;\QQ)$ is also finite dimensional. On the other hand, by
Proposition \ref{Q}, if $H_{\sfG}^*(\calZ_K,\QQ)$ has no odd degree, then $\QQ[K]$ is free over $\QQ[\sfR^*]$. Since
$H_{\sfG}^*(\calZ_K;\QQ) \cong H_{\sfT}^*(\calZ_K;\QQ) \otimes_{\QQ[\sfR^*]}\QQ$, the finiteness of
$H_{\sfG}^*(\calZ_K;\QQ)$ implies that $H_{\sfT}^*(\calZ_K;\QQ)$ is finitely generated over $\QQ[\sfR^*]$.
\end{proof}
\subsection{Torsion-freeness}
First we observe the following equivalence.
\begin{lem}\label{Q=Z}
Then $\ZZ[K]$ is torsion-free over $\ZZ[\sfR^*]$ if and only if $\ZZ[K]\otimes\QQ$ is torsion-free over $\QQ[\sfR^*]$.
\end{lem}
\begin{proof}
Suppose that $f\not=0$ is a torsion element in $\ZZ[K]$ over $\ZZ[\sfR^*]$, i.e. there is $g \in \ZZ[\sfR^*]$ such that $fg=0$
in $\ZZ[K]$. Since $\ZZ[K]$ is free over $\ZZ$, $f\not=0$ in $\QQ[K]$. Therefore $f$ is also a torsion in $\QQ[K]$ over
$\QQ[\sfR^*]$. On the other hand, suppose that $f\not=0$ is a torsion element of $\QQ[K]$ over $\QQ[\sfR^*]$. Let $g \in
\QQ[\sfR^*]$ such that $fg =0$ in $\QQ[K]$. Let $a$ be the product of denominators of the coefficients of $f$ and $b$ be
the product of denominator of coefficients of $g$. Then the pair of $af \in \ZZ[K]$ and $bg \in \ZZ[\sfR^*]$ gives a
torsion of $\ZZ[K]$ over $\ZZ[\sfR^*]$.
\end{proof}
\begin{thm} \label{surj=>torsionfree}
If $H_{\sfG}^*(\calZ_K;\QQ)$ has no odd degree classes, then $\ZZ[K]$ is torsion-free over $\ZZ[\sfR^*]$.
\end{thm}
\begin{proof} By Proposition \ref{Q}, $\QQ[K]$ is free over $\QQ[\sfR^*]$, therefore it is torsion-free over $\QQ[\sfR^*]$. We conclude that $\ZZ[K]$ is torsion-free over $\ZZ[\sfR^*]$ by Lemma \ref{Q=Z}.
\end{proof}
If $[\calZ_K/\sfG]$ is a quasi-toric orbifold or a complete toric DM stack (Section \ref{bg}), $H_{\sfG}^*(\calZ_K;\QQ)$ has no odd degree classes by Theorem \ref{qtQ} and \ref{dmsQ}. Thus we have
\begin{cor}\label{cortf}
If $[\calZ_K/\sfG]$ is a quasi-toric orbifold or a complete toric DM stack, then $\ZZ[K]$ is torsion-free over $\ZZ[\sfR^*]$.
\end{cor}
\begin{rem}
The converse of Theorem \ref{surj=>torsionfree} is not true.  The direct product of weighted projective spaces is a complete toric DM stack and its cohomology has odd degree classes. See Example \ref{nonexm}.
\end{rem} 
The following proposition shows that the vanishing of odd classes in $H_{\sfG}^*(\calZ_K;\QQ)$ implies that the size of $\sfG$ is maximal and it is analogous to Proposition 5.2 \cite{Franz09}.
\begin{prop}\label{torsionmax}
 Let $n'$ be the largest cardinality of a face of $K$. Suppose that $\sfG$ acts on $\calZ_K$ locally freely. If
 $H_{\sfG}^*(\calZ_K;\QQ)$ has no odd degree, then $\dim \sfR = n'$. Furthermore for any subgroup $\sfU \subset \sfG$
 such that $\dim \sfU < \dim \sfG$, $\QQ[K]$ has a torsion over $\QQ[\tilde{\sfR}^*]$ where $\tilde{\sfR}:=\sfT/\sfU$.
\end{prop}
\begin{proof}
By Proposition \ref{fgQ}, $\QQ[K]$ is finitely generated over $\QQ[\sfR^*]$ and hence over $\QQ[\tilde{\sfR}^*]$. Moreover $\QQ[K]$ is free over $\QQ[\sfR^*]$ by Proposition \ref{Q} and so $\Ann_{\QQ[\sfR^*]}\QQ[K]=0$. Thus we have
\[
n'=\dim \QQ[K] = \dim \QQ[\sfR^*] =\dim \frac{\QQ[\tilde{\sfR}^*]}{\Ann_{\QQ[\tilde{\sfR}^*]} \QQ[K]}.
\]
Thus $\dim \sfR=\dim \QQ[\sfR^*]=n'$ and $\dim \QQ[\tilde{\sfR}^*] = \dim \tilde{\sfR} > \dim \sfR$ implies $\Ann_{\QQ[\tilde{\sfR}^*]} \QQ[K]\not=0$.
\end{proof}
The above proposition doesn't allow us to construct a torsion element explicitly. Below we show a way to find one for the case of toric manifolds .
\subsubsection{How to find a torsion element}
We will use the GKM description of the equivariant cohomology of a toric manifold. Let $\Delta$ be an $n$-dimensional
Delzant polytope, i.e. a labeled polytope associated to a toric manifold. Let $H_1,\cdots, H_m$ be the facets of
$\Delta$. As in Section \ref{PoddarSarkar} and \ref{LermanTolman}, $\sfT=\U(1)^m$, $\sfR=\U(1)^n$ and $B:\sfT \to \sfR$
is given by the $n\times m$ integral matrix $B:=[\beta_1,\cdots,\beta_m]$ where $\beta_i$ is the inward primitive normal
vector for $H_i$. We adopt Notation \ref{notation}. Let $v_1,\cdots, v_d$ be the vertices of $\Delta$ and let $e_{ij}$
be the edge connecting $v_i$ and $v_j$.

The $\sfR$-equivariant cohomology of $\sfX_{\Delta}$ coincides with the $\sfT$-equivariant cohomology of
$\calZ_{\Delta}$ and it is the Stanley-Reisner ring $\ZZ[K_{\Delta}]$ for the associated simplicial complex
$K_{\Delta}$. Since $\Delta$ is Delzant, $\ZZ[K_{\Delta}]$ is free over $\ZZ[\sfR^*]$ (c.f. Theorem \ref{DJsmooth},
Proposition \ref{free1}).

The injectivity theorem of the Hamiltonian $\sfR$-action on $\sfX_{\Delta}$ states that pulling back to fixed points gives the injective map of $\QQ[\sfR^*]$-algebras:
\[
\xymatrix{
\QQ[K_{\Delta}] \ar[r]^{\Phi\ \ \ \ \ \ \ \ \ \ } & \bigoplus_{i=1}^d \QQ[u_1,\cdots, u_n]
}
\]
and the GKM theorem states that the image of $\Phi$ is given by
\[
\mbox{GKM}(\Delta):=\left\{\left.\sff:=(\sff_1,...,\sff_d)\in\bigoplus_{i=1}^{d}\QQ[u_1,...,u_n]\right|\ \alpha_{ij}| (\sff_{i}-\sff_{j})\ \mbox{for each edge}\ e_{ij}\right\}
\]
where $\alpha_{ij}$ is a linear polynomial in $u_1,...,u_n$ associted to each edge $e_{ij}$ in the GKM theorem.
The following proposition describes the image of $x_i \in \QQ[K_{\Delta}]$ explicitly in terms of the matrix $B$.
\begin{prop} \label{GKM}
Let $v:=H_{i_1}\cap \cdots \cap H_{i_n}$ be a vertex of $\Delta$ and $\sigma_v=\{i_1<\cdots<i_n\}$. Let $B_v:=[\beta_{i_1},\cdots, \beta_{i_n}]$. Let $\alpha^v_r$ be the $r$-th row of $B_v^{-1}$. Then
\[
\Phi(x_{i})_v = 
\begin{cases}
\alpha^v_r \cdot (u_1,\cdots,u_n)^T & \mbox{ if } i = i_r \mbox{ for some } r=1,\cdots,n\\
0 & \mbox{ if otherwise}.
\end{cases}
\]
\end{prop}
\begin{proof}
Recall that $x_i$ is the first Chern class of the pullback of the circle bundle $L_i:=E\sfT\times_{\sfT} \sfT_i \to
B\sfT$ along $E\sfT\times_{\sfT} \calZ_{\Delta} \to B\sfT$ where $\sfT_i=\U(1)^{\{i\}} \times \{1\}^{[m]\backslash
  \{i\}} \subset \sfT$ (see Notation \ref{notation0}). The Borel space of the fixed point is $E\sfT\times_{\sfT} ((\sfT
\times v)/\sfT_{\sigma_v})$ as a subspace of $E\sfT \times_{\sfT} \calZ_{\Delta}$ (See the definition of
$\calZ_{\Delta}$ in Section \ref{PoddarSarkar}). Thus $\Phi(x_i)_v$ is the first Chern class of the pullback of the line
bundle $L_i$ along the projection $\pi: E\sfT\times_{\sfT} ((\sfT \times v)/\sfT_{\sigma_v}) \to B\sfT$. Since $L_i
\cong B\sfT_{[m]\backslash \{i\}} \times E\sfT_i$ and $E\sfT\times_{\sfT} ((\sfT \times v)/\sfT_{\sigma_v})
\cong  B\sfT_{\sigma_v} \times E\sfT_{[m]\backslash \sigma_v}$, it is clear that, if $i \not \in \sigma_v$, then $\pi^*L_i$ is trivial.

Now suppose that $i=i_r$ for some $r=1,\cdots,n$. Recall from Notation \ref{notation} that $u_k=\sum_{j=1}^m B_{kj}x_j$
where $B=(B_{kj})$. We observe that, for each for each $k=1,\cdots,n$, the $i_k$-th column of $B_v^{-1}\cdot B$ is
$(0,\cdots,0,1, 0\cdots,0)^T$ where $1$ is located at $k$-th entry. Therefore the $r$-th row of $B_v^{-1} \cdot B \cdot
(x_1,\cdots,x_m)^T$ is $x_{i_r} + \sum_{j \in [m]\backslash \sigma_v}^m c_jx_j$ for some integers $c_j$. Thus
\[
\alpha^v_r \cdot (u_1,\cdots,u_n)^T=\alpha^v_r \cdot B \cdot (x_1,\cdots,x_m)^T = x_i + \sum_{j \in [m]\backslash \sigma_v}^m c_jx_j.
\]
From the previous argument, we have $\Phi(x_i)|_v=\Phi(x_i+\sum_{j \in [m]\backslash \sigma_v}a_jx_j)|_v$ for any
$a_j\in \QQ$. Also since $\Phi$ is a $\QQ[\sfR^*]$-algebra homomorphism, $\Phi(u)=(u,\cdots,u)$ for any linear
combination $u$ of $u_1,\cdots, u_n$. Thus
\[
\Phi(x_i)|_v=\left.\Phi\left(x_i+\sum_{j \in [m]\backslash \sigma_v}c_jx_j\right)\right|_v = \left.\Phi\left(  \alpha^v_r \cdot (u_1,\cdots,u_n)^T \right)\right|_v = \alpha^v_r \cdot (u_1,\cdots,u_n)^T.
\]
\end{proof}
\begin{rem}
The row vector $\alpha^v_r$ is perpendicular to $\beta_{i_k}, k\not=r$ and is parallel to the edge $\cap_{k\not=r}
H_k$. Furthermore since the inner product of $\alpha^v_r$ and $\beta_{i_r}$ is $1$, we can conclude that it is the
primitive vector parallel to the edge and pointing out from $v$ along the edge.
\end{rem}
\begin{cor}\label{torxx}
Using the notation above, let $\sfU \subset \sfG$ be a subtorus of dimension $\dim \sfG -1$. Let
$\tilde{\sfR}:=\sfT/\sfU$ and $\ZZ[\tilde{\sfR}^*]=\ZZ[u_1,\cdots, u_n, u_{n+1}]$. Let $v=H_{i_1}\cap \cdots \cap
H_{i_n}$ be a vertex of $\Delta$. Then $x_{i_1}\cdots x_{i_n}$ is a torsion element in $\QQ[K_{\Delta}]$ over
$\QQ[\tilde{\sfR}^*]$.
\end{cor}
\begin{proof}
First note that $x_{i_1}\cdots x_{i_n}$ is a none zero element in $\ZZ[K]$ and, by the above proposition,
$\Phi(x_{i_1}\cdots x_{i_n})|_w=0$ for all $w\not=v$. Thus we have $(u_{n+1}-\sum_{i=1}^{n}a_{i}u_{i}) \cdot
\Phi(x_{i_1}\cdots x_{i_n})=0$ if we let $\Phi(u_{n+1})|_{v}=\sum_{i=1}^{n}a_{i}u_{i}$. By the injectivity theorem,
$(u_{n+1}-\sum_{i=1}^{n}a_{i}u_{i})\cdot (x_{i_1}\cdots x_{i_n})=0$. Since $u_{n+1}-\sum_{i=1}^{n}a_{i}u_{i}$ is
none-zero by definition, we conclude that $x_{i_1}\cdots x_{i_n}$ is a torsion element over $\ZZ[\tilde{\sfR}^*]$.
\end{proof}
\begin{exm}
Let $\Delta$ be the Delzant polytope that is the unit square where $H_i$'s are facets and $v_j$'s are vertices: 
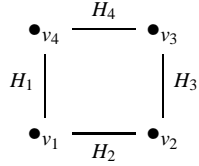
\[
\xymatrix{
\bullet_{v_4} \ar@{-}[r]^{H_4}\ar@{-}[d]_{H_1}&\bullet_{v_3}\ar@{-}[d]^{H_3} \\
\bullet_{v_1} \ar@{-}[r]_{H_2}&\bullet_{v_2} 
}
\]
Let $\sfT=\sfU(1)^4$, $\sfR=\sfU(1)^2$ and the map $B:\sfT \to \sfR$ given by the matrix $\left(\begin{array}{cccc}1 & 0
  & -1 & 0 \\ 0 & 1 & 0 & -1 \\ \end{array}\right)$. The kernel of $B$ is $\sfG= \{(t,s,t^{-1},s^{-1})\}$. The
corresponding toric manifold is $\CP^1\times \CP^1$. By Proposition \ref{GKM}, the injectivity map $\Phi:
\QQ[K_{\Delta}] \to \bigoplus_{i=1}^{4}\QQ[u_1,u_2]$ is given by
\[
\Phi(x_1)=(u_1,0,0,u_1),\ \ \  \Phi(x_2)=(u_2,u_2,0,0), \ \ \ \Phi(x_3)=(0,-u_1,-u_1,0),\ \ \ \Phi(x_4)=(0,0,-u_2,-u_2).
\]
where $u_1=x_1-x_3,u_2=x_2-x_4$.
The GKM theorem states that the image can be described by GKM condition:
\[
\im(\Phi)=\left\{ \left.(f_1,f_2,f_3,f_4) \in\bigoplus_{i=1}^{4}\QQ[u_1,u_2]\    \right|\   u_1|f_1-f_2, \ \ u_2|f_2-f_3, \ \ u_1|f_3-f_4, \ \ u_2|f_4-f_1 \right\}.
\]
Let $\sfU:=\{(1,s,1,s^{-1}\} \subset \sfG$ be the subtorus and $\tilde{\sfR}=\sfT/\sfU$.  By Proposition
\ref{torsionmax}, $\QQ[K_{\Delta}]$ has torsion elements as a module over $\QQ[\tilde{\sfR}^{*}] =\QQ[u_1,u_2,u_3]$,
where we can take $u_3=x_2+x_3-x_4$ and so $\Phi(u_{3})=(u_2,-u_1+u_2,-u_1+u_2,u_2)$.

By Corollary \ref{torxx}, $x_1x_2$ is a torsion element. Indeed, $\Phi(x_1x_2) = (u_1u_2,0,0,0)$. Consider $u_3-u_2 \in
\QQ[\tilde{\sfR}^*]$. We have $\Phi(u_3-u_2)=(0,-u_1,-u_1,0)$. Since $(0,-u_1,-u_1,0)\cdot (u_1u_2,0,0,0) =(0,0,0,0)$
and by the injectivity, $(u_3-u_2)\cdot x_1x_2 = 0$ in $\QQ[K_{\Delta}]$. Thus $x_1x_2 \in \QQ[K_{\Delta}]$ is a torsion
element over $\QQ[\tilde{\sfR}^*]$. It is not hard to see what are doing here works equally well for general Delzant
polytope $\Delta$ in any dimension.
\end{exm}
\subsection{An injectivity theorem and freeness}
Suppose that a connected subtorus $\sfG$ of $\sfT$ acts on $\calZ_K$ locally freely and consider the a torus $\sfW$ such
that $\sfG \varsubsetneq\sfW \subset \sfT$ and $m':=\dim \sfW$.  Let $F \subset \calZ_K$ be the set of
$(m-n)$-dimensional $\sfW$-orbits. In this section, we discuss the injectivity of
\[
H_{\sfW}(\calZ_K;\ZZ) \to H_{\sfW}(F;\ZZ).
\] 
We have the following injectivity result for a $\sfW$-action on $\calZ_K$ when $[\calZ_K/\sfG]$ is a symplectic compact toric orbifold. 
\begin{thm}\label{Injectivity}
Suppose that $[\calZ_K/\sfG]$ is a symplectic compact toric orbifold corresponding to a labeled polytope $(\Delta,
\sfb)$. Suppose that the stabilizer of any point $x \in \calZ_K$ in $\sfW$ is connected. Then $H_{\sfT}(\calZ_K;\ZZ) \to
H_{\sfW}(F;\ZZ)$ is injective. In particular, $H_{\sfT}(\calZ_K;\ZZ)$ is free over $\ZZ[(\sfT/\sfW)^*]$.
\end{thm}
\begin{proof}
Let $\{H_1,\cdots, H_m\}$ be the set of all facets of $\Delta$. The simplicial complex $K$ is the one associated to
$\Delta$ and $\tau \in K$ iff $\Delta_{\tau} := \cap_{i\in\tau}H_i \not=\varnothing$. Let $F_a$ be a connected component
of $F$. Then by Lemma 3.4 \cite{HM} and our assumption of connected isotropy groups, the stabilizers of every $x \in
F_a$ in $W$ coincide. Let $W_a$ be the stabilizer of points in $F_a$. Let $\mu: \calZ_K \to \Delta$ be the moment map.
Note that this is the quotient map by the action of $\sfT$. First we show
\begin{lem}\label{sub1}
$F_a = \mu^{-1}(\Delta_{\sigma})$ for some $\sigma \in K$. In particular, the stabilizer of each point $x \in F_a$ is $\sfW_{\sigma}:=\sfW \cap \sfT_{\sigma}$.
\end{lem}
 \emph{Proof of Lemma \ref{sub1}}: For $x \in F_a$, let $\sigma_x:=\{ i \in [m], x_i = 0\}$. Then
 $\mu^{-1}(\Delta_{\sigma_x}) \subset F_a$ and the unique stabilizer for $F_a$ is given by $\sfW_a=\sfW \cap
 \sfT_{\sigma_x}$. Note that $\sigma_x\not=\varnothing$ since $m >n$. It suffices to show that there is an element $x
 \in F_a$ such that $\sigma_x$ is the unique minimal subset among the collection of subsets, $\{ \sigma_y \ |\ y\in
 F_a\}$.  Let $\sigma_x$ and $\sigma_y$ be minimal for some $x,y\in F_a$. Suppose that $\sigma_x \not = \sigma_y$ and
 consider $z \in \mu^{-1}(\Delta_{\sigma_x \cap \sigma_y})$ such that $\sigma_z = \sigma_x \cap \sigma_y$. Since
 $W_a=\sfW \cap \sfT_x \cap \sfT_{\sigma_y}=\sfW_z$, $z \in F$ by dimension counting.  The connectivity of $F_a$
 then implies that $z \in F_a$. This contradicts to the assumption that $\sigma_x$ and $\sigma_y$ are minimal, so
 $\sigma_x=\sigma_y$. \qed

Now let $\{F_{\sigma}\}$ be the set of connected components of $F$ where $F_{\sigma} = \mu^{-1}(\Delta_{\sigma})$. For
each $\sigma$, choose a splitting $\sfW=\sfW_{\sigma} \times (\sfW/\sfW_{\sigma})$. The target of the injectivity map is
computed as follows:
\[
H_{\sfW}(F,\ZZ) = \bigoplus_{\sigma} H_{\sfW} (F_{\sigma}, \ZZ) = \bigoplus_{\sigma} H(B\sfW \times_{\sfW} F_{\sigma}, \ZZ) = \bigoplus_{\sigma} H(BW_{\sigma},\ZZ) \otimes H(F_{\sigma}/(\sfW/\sfW_{\sigma}), \ZZ). 
\]
Now we show that $F_{\sigma}/(\sfW/\sfW_{\sigma})$ is a compact toric symplectic manifold so that $H(F_{\sigma}/(\sfW/\sfW_{\sigma}), \ZZ)$ has only even degree and has no $\ZZ$-torsion. Since $[F_{\sigma}/\sfG]$ is the suborbifold of $\sfW/\sfG$-fixed orbifold points, it is a symplectic orbifold (c.f. \cite[p.4210, Cor 3.8]{LT}), which is compact. Since the unique stabilizer of points of $F_{\sigma}$ in $\sfG$ is given by $\sfG_{\sigma}=\sfG \cap \sfW_{\sigma}$, $F_{\sigma}/(\sfG/\sfG_{\sigma})$ is a compact toric manifold with the effective Hamiltonian action of $(\sfT/\sfG)/(\sfT_{\sigma}/\sfG_{\sigma})$. On the other hand, $F_{\sigma}/(\sfG/\sfG_{\sigma})$ is exactly $F_{\sigma}/\sfW = F_{\sigma}/(\sfW/\sfW_{\sigma})$. Thus $H(F_{\sigma}/(\sfW/\sfW_{\sigma}), \ZZ)$ has only even degree and has no $\ZZ$-torsion. Now the injectivity of $H_{\sfW}(\calZ_K,\ZZ) \to H_{\sfW}(F,\ZZ)$ implies that $H_{\sfW}(\calZ_K,\ZZ)$ has no $\ZZ$-torsion and has only even degree. It also implies the freeness of $H_{\sfT}(\calZ_K,\ZZ)$ over $\ZZ[(\sfT/\sfW)^*]$ by Theorem \ref{free1}.

To apply the injectivity theorem over $\ZZ$ (Remark 4.10 \cite{HM}), we need to have that $\sfW_{\sigma}$ is connected
and the weights of the action on the (negative) normal bundle are all primitive for each connected component of
$F_{\sigma}$. The former is true by the assumption. For the latter, look at the normal bundle of $F_{\sigma}$ in
$\calZ_K$ which is given by $\bigoplus_{i\in \sigma} \CC\frac{\partial}{\partial z_i}$. The weights of the
$\sfT_{\sigma}$-equivariant normal bundle are the standard $\ZZ$-basis $\{\lambda_i, i\in\sigma\}$ of
$\sfN_{\sfT_{\sigma}}^*$. We need to check that the induced $W_{\sigma}$-weights $\tilde{\lambda}_i:=
A_{\sigma}^*(\lambda_i) \in \sfN_{\sfW_{\sigma}}^*$ are non-zero and primitive where $A_{\sigma}:W_{\sigma} \inc
T_{\sigma}$ is the restriction of the natural inclusion $A: \sfW \to \sfT$. It is easy to see that $\tilde{\lambda_i}$
is non-zero, since, if otherwise, the normal direction $\CC\frac{\partial}{\partial z_i}$ is also contained in
$F_a$. Finally the proof is completed by the following lemma.
\begin{lem}\label{prim}
$A_{\sigma}^*(\lambda_i) \in \sfN_{\sfW_{\sigma}}^*$ is primitive.
\end{lem}
\emph{Proof of Lemma \ref{prim}}: Consider the following commutative diagram of tori and its dual for the weight lattices:
\begin{equation}
\xymatrix{
\sfW_{\sigma\backslash\{i\}}\ar[d] \ar[r]& \sfT_{\sigma\backslash\{i\}}\ar[d] \\
\sfW_{\sigma}\ar[r]^{A_{\sigma}}\ar[rd]_{\sff_i} & \sfT_{\sigma}\ar[d]^{\sfg_i}\\
& \sfT_{\{i\}}
}
\ \ \ \ \ \ \  \ \ \ \ \ \ \ \ 
\xymatrix{
\sfN_{\sfW_{\sigma\backslash\{i\}}}^* &\sfN_{\sfT_{\sigma\backslash\{i\}}}^*\ar[l]\\
\sfN_{\sfW_{\sigma}}^* \ar[u]&\sfN_{\sfT_{\sigma}}^*\ar[l]_{A_{\sigma}^*}\ar[u]\\
& \sfN_{\sfT_{\{i\}}} \ar[u]_{\sfg_i^*}\ar[lu]^{\sff_i^*}
}
\end{equation}
Here $\sfg_i$ is the canonical projection. The map $\sff_i = \sfg_i\circ A_{\sigma}$ must be surjective since $\sfT_i$
is one dimensional and $\tilde{\lambda_i}$ is non-zero. Also we have $W_{\sigma\backslash\{i\}} = \ker \sff_i$ which
must be connected by the assumption.  Therefore we have a short exact sequence tori $0 \to \sfW_{\sigma\backslash\{i\}}
\to \sfW_{\sigma} \to \sfT_{\{i\}} \to 0$ which implies that $f_i^*$ maps $N_{\sfT_i}^*$ to a direct summand. Thus
$\tilde{\lambda}_i$ must be primitive since $\lambda_i$ is a basis of $\sfN_{\sfT_i}^*$.
\end{proof}
\begin{rem}
Theorem \ref{Injectivity} holds when $\sfG=\sfW$. In this case, by the assumption, $\sfG$ acts on $\calZ_K$ is free and
$F = \calZ_K$. We recover the fact that the equivariant cohomology of toric manifolds (or smooth toric varieties) is
free over $\ZZ[\sfR^*]$.
\end{rem}
\section{{\bf Examples}}\label{exm}
\begin{exm}[Effective Weighted Projective Spaces]
Let $\sfa:=(a_1,\cdots, a_m)$ be a sequence of positive integers with $\gcd(a_1,\cdots, a_m)\not=1$ and let
$[\CP^{m-1}_{\sfa}]$ be the corresponding effective weighed projective space. As in Section \ref{bg},
$H^*([\CP^{m-1}_{\sfa}];\ZZ) = H_{\sfG}^*(\calZ_K;\ZZ)$ where $\sfG = \{(t^{a_1},\cdots,t^{a_m})\} \subset \sfT$ and $K$
the boundary of an $(m-1)$-simplex. Here we will not write the matrix $B$. It would be useful if there is a formula to
describe a $\ZZ$-basis of the dual weight lattice of $\sfR$ in terms of $(a_1,\cdots, a_m)$ but to our knowledge it is
not known.

The corresponding Stanley-Reisner ring is $\ZZ[x_1,\cdots, x_m]/\lan x_1\cdots x_m\ran$. In \cite{Holm08}, the ordinary cohomology is computed:
\[
H^*_{\sfG}(\calZ_K,\ZZ) \cong \ZZ[y]/\lan a_1\cdots a_m y^m\ran.
\]
It has only even degree, so $H^*_{\sfT}(\calZ_K,\ZZ) \to H_{\sfG}^*(\calZ_K,\ZZ)$ is surjective. However
$H_{\sfG}^*(\calZ_K,\ZZ)$ has $\ZZ$-torsion, and so $H^*_{\sfT}(\calZ_K,\ZZ)$ is not free over $\ZZ[\sfR^*]$.

Thus $H^*_{\sfT}(\calZ_K,\ZZ) \to H_{\sfG}^*(\calZ_K,\ZZ)$ is surjective, but the source is not free over
$\ZZ[\sfR^*]$. $[\CP^1_{12}]$ is the simplest such example where the dimension is 2. This provides an answer to a question
analogous to Question 1.1 \cite{FranzPuppeOsaka}, see also \cite{Allday}.
\end{exm}
\begin{exm}\label{nonexm}
As we saw in the previous examples, $H_{\sfG}(\calZ_K,\ZZ)$ has $\ZZ$-torsion (infinitely many) in even degree for
$[\calZ_K/\sfG]=\CP^{m-1}_{\sfa}$. The direct product of such toric orbifolds is also a toric orbifold and the
K\"{u}nneth theorem provides the $\ZZ$-torsions in odd degree. More concretely, take the labeled polytope
\[
\xymatrix{
\bullet\ar@{-}[r]^1_{H_4}\ar@{-}[d]_1^{H_1}&\bullet\ar@{-}[d]^2_{H_3}\\
\bullet\ar@{-}[r]_2^{H_2} & \bullet 
}
\]
which gives $B=\begin{pmatrix} 1&0 &-2 & 0 \\ 0 &2 & 0 &-1 \end{pmatrix}$. This defines the direct product
$[\CP^1_{12}\times \CP^1_{12}]$ which gives odd degree elements in $H_{\sfT}(\calZ_K,\ZZ)$. Thus
$H^*_{\sfT}(\calZ_K,\ZZ) \to H_{\sfG}^*(\calZ_K,\ZZ)$ is not surjective. We can also see this by checking if $(x_1-
2x_3, 2x_2 - x_4)$ is a regular sequence of $\ZZ[K]=\frac{\ZZ[x_1,\cdots, x_4]}{\lan x_1x_3, x_2x_4\ran}$ as a module
over $\ZZ[x_1- 2x_3, 2x_2 - x_4]$. Indeed, it is not a regular sequence: $x_1- 2x_3$ is a non-zero divisor in $\ZZ[K]$
but $2x_2 - x_4$ is a zero divisor in $\ZZ[K]/(x_1- 2x_3)$ since $(2x_2 - x_4)x_2x_3^2 = 0$ and $x_2x_3^2\not=0$ in
$\ZZ[K]/(x_1- 2x_3)$.

From this example, we can create more examples by the method of the symplectic cut. Consider the cutting by a hyperplane $H_5$:
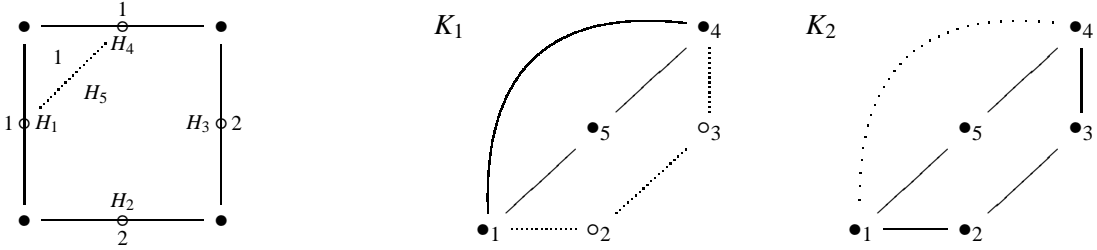
\[
\xymatrix{
\bullet\ar@{-}[rr]^1_{H_4}\ar@{-}[dd]_1^{H_1}&\circ&\bullet\ar@{-}[dd]^2_{H_3}\\
\circ \ar@{.}[ru]_{H_5}^1&&\circ\\
\bullet\ar@{-}[rr]_2^{H_2} &\circ& \bullet 
} \ \ \ \ \  \ \ \ \ \ \ \ \ \ \ \ \ \ \ \ \ \ \ \ \ K_1
\xymatrix{
& & \bullet_4 \ar@{.}[d]\ar@{-}[ld]\\
&\bullet_5 \ar@{-}[ld]& \circ_3\\
\bullet_1 \ar@/^3pc/@{-}[rruu] \ar@{.}[r]& \circ_2\ar@{.}[ru] & 
}
\ \ \ \ \ \ \ \ \ \ 
K_2
\xymatrix{
& & \bullet_4 \ar@{-}[d]\ar@{-}[ld]\\
&\bullet_5 \ar@{-}[ld]& \bullet_3\\
\bullet_1 \ar@/^3pc/@{.}[rruu] \ar@{-}[r]& \bullet_2\ar@{-}[ru] & 
}
\]
where $K_1$ and $K_2$ are simplicial complexes associated to the cut pieces.  Let $\tilde{B}=\begin{pmatrix} 1&0 &-2 & 0
& 1 \\ 0 &2 & 0 &-1&-1 \end{pmatrix}$. Then $\tilde{B}$ defines a 2-dimensional subtorus $\sfG$ of 5 dimensional torus
$\sfT$, which acts on $\calZ_{K_1}$ and $\calZ_{K_2}$ locally freely. Each $[\calZ_{K_1}/\sfG]$ and $[\calZ_{K_2}/\sfG]$
defines the toric orbifolds corresponding to the symplectic cut of $\CP^1_{12}\times\CP^1_{12}$. The surjectivity holds
for $K_1$ but not for $K_2$. We can see this by checking that, if $u_1=x_1-2x_3 + x_4$ and $u_2= 2x_1-x_4-x_5$, then
$(u_1,u_2)$ is a regular sequence of $\ZZ[K_1]$, while $(u_1,u_2)$ is not a regular sequence for $\ZZ[K_2]$ since
$2x_1-x_4-x_5$ is annihilated by $x_2x_3^2$. By the same algebraic computation, we see that $(u_1,u_2)$ is not a regular
sequence of $\ZZ[K_1\cup K_2]$ as a $\ZZ[u_1,u_2]$-module, i.e. $H_{\sfT}(\calZ_{K_1\cup K_2}, \ZZ) \to
H_{\sfG}(\calZ_{K_1\cup K_2}, \ZZ)$ is not surjective.
\end{exm}

\section{{\bf Algebraic Gysin Sequence}} \label{secgysin}
Let $\sfU$ be a subgroup of $\sfG$ such that $\sfL:=\sfG/\sfU$ is a $1$-dimensional torus. We have a principal $\sfL$-bundle $\pi: E\sfT \times_{\sfU} \calZ_K \to E\sfT\times_{\sfG} \calZ_K$ and the corresponding Gysin sequence
\[
\cdots \to H^{i-1}_{\sfG}(\calZ_K,\ZZ)\stackrel{\cup \bfe}{\longrightarrow}  H^{i+1}_{\sfG}(\calZ_K,\ZZ) \stackrel{\pi^*}{\longrightarrow} H^{i+1}_{\sfU}(\calZ_K,\ZZ)  \stackrel{\pi_*}{\longrightarrow} H^i_{\sfG}(\calZ_K,\ZZ) \stackrel{\cup\bfe}{\longrightarrow}  H^{i+2}_{\sfG}(\calZ_K,\ZZ) \to \cdots  
\]
where $\bfe$ is the Euler class of the bundle and $\pi_*$ / $\pi^*$ is the pushforward / pullback map. In the light of
Theorem \ref{Franz}, it is natural to ask if there is a purely algebraic construction of a long exact sequence of
$\Tor$'s corresponding the Gysin sequence. We describe the construction in the following
\begin{construction}[Algebraic Gysin Sequence]\label{GS}
Let $\tilde{\sfR}:=\sfT/\sfU$ and identify $H^*(B\sfR,\ZZ)=\ZZ[u_1,\cdots,u_n,u_{n+1}]=\ZZ[\tilde{\sfR}^*]$ where
$\ZZ[u_1,\cdots,u_n]=H^*(B\sfR,\ZZ)$. Consider the short exact sequence of Koszul complexes for (as modules over
$\ZZ[\tilde{\sfR}^*]$):
\begin{equation}\label{gysin}
0 \to K^{\ZZ[\tilde{\sfR}^*]}(\xi_1,\cdots,\xi_n) \stackrel{\tau^*}{\longrightarrow} K^{\ZZ[\tilde{\sfR}^*]}(\xi_1,\cdots,\xi_n,\xi_{n+1}) \stackrel{\tau_*}{\longrightarrow} K^{\ZZ[\tilde{\sfR}^*]}(\xi_1,\cdots,\xi_n)[-1] \to 0
\end{equation}
where the first map is the obvious inclusion denoted by $\tau^*$ and the second map is \emph{getting rid of
  $\xi_{n+1}\wedge $} denoted by $\tau_*$. Note that $K^{\ZZ[\tilde{\sfR}^*]}(\xi_1,\cdots,\xi_n) =
\ZZ[u_1,\cdots,u_n,u_{n+1}]\lan \xi_1,\cdots, \xi_n\ran$ and the differential is giving by extending $\partial \xi_i =
u_i$ as a differential algebra where $\lan \ \ \ran$ denotes the exterior algebra.

Let $\sfM$ be a graded $\ZZ[x_1,\cdots,x_m]$-module. After tensoring $\sfM$, we obtain the long exact sequence of Tor modules over $\ZZ[u_1,\cdots, u_n,u_{n+1}]$:
\begin{equation}\label{algGysin}
\cdots \to \Tor_{i+1}^{\ZZ[\sfR^*]}(\sfM, \ZZ)\stackrel{\delta}{\longrightarrow} \Tor_{i+1}^{\ZZ[\sfR^*]}(\sfM, \ZZ) \stackrel{\tau^*}{\longrightarrow}
\Tor_{i+1}^{\ZZ[\tilde{\sfR}^*]}(\sfM,\ZZ) \stackrel{ \tau_*}{\longrightarrow} \Tor_{i}^{\ZZ[\sfR^*]}(\sfM, \ZZ)\stackrel{\delta }{\longrightarrow} \Tor_{i}^{\ZZ[\sfR^*]}(\sfM, \ZZ)  \to \cdots 
\end{equation}
We call this the \emph{algebraic (homological) Gysin sequence}.
\end{construction}
\begin{prop}
The connecting map $\delta$ is a multiplication by $u_{n+1}$ and it is independent of the choice of $u_{n+1}$.
\end{prop}
\begin{proof}
It follows from the diagram chasing. Consider the part of the map of complexes
\[
\xymatrix{
K_{i-1}^{\ZZ[\tilde{\sfR}^*]}(\xi_1,\cdots,\xi_n) \ar[r] & K_{i-1}^{\ZZ[\tilde{\sfR}^*]}(\xi_1,\cdots,\xi_n) \ar[r] &K_{i}^{\ZZ[\tilde{\sfR}^*]}(\xi_1,\cdots,\xi_n) \\
K_i^{\ZZ[\tilde{\sfR}^*]}(\xi_1,\cdots,\xi_n) \ar[u]_{\partial}\ar[r] & K_i^{\ZZ[\tilde{\sfR}^*]}(\xi_1,\cdots,\xi_n)\ar[u]_{\partial} \ar[r] &K_{i-1}^{\ZZ[\tilde{\sfR}^*]}(\xi_1,\cdots,\xi_n) \ar[u]_{\partial}
}
\]
Let $z$ be a cycle in $K_{i-1}^{\ZZ[\tilde{\sfR}^*]}(\xi_1,\cdots,\xi_n)$ and lift it to top left corner:
\[
\xymatrix{
u_i\cdot z \ar[r] & \partial ( \xi_{n+1} \wedge z) = u_i\cdot z\ar[r] &0 \\
 &  \xi_{n+1}\wedge z\ar[r] \ar[u]_{\partial}&z \ar[u]_{\partial}
}.
\]
Since Tor modules are independent of the choice of the basis of the polynomial ring, so $\delta$ is independent of the choice of $u_{n+1}$.  
\end{proof}
\begin{defn}[Cohomological Algebraic Gysin sequence]
As in Remark \ref{cohdegree}, we can assign the cohomological degree and turn the sequence (\ref{algGysin}) into a cohomological sequence: 
\begin{eqnarray}\label{cohAlgGysin}
\cdots \to \Tor^{i-1}_{\ZZ[\sfR^*]}(\sfM, \ZZ)\stackrel{\cdot u_{n+1}}{\longrightarrow} \Tor^{i+1}_{\ZZ[\sfR^*]}(\sfM, \ZZ) \stackrel{\tau^*}{\longrightarrow}
\Tor^{i+1}_{\ZZ[\tilde{\sfR}^*]}(\sfM,\ZZ) \stackrel{ \tau_*}{\longrightarrow} \Tor^{i}_{\ZZ[\sfR^*]}(\sfM, \ZZ)\stackrel{\cdot u_{n+1}}{\longrightarrow} \Tor^{i+2}_{\ZZ[\sfR^*]}(\sfM, \ZZ)  \to \cdots 
\end{eqnarray}
We call this the \emph{cohomological algebraic Gysin sequence}.
\end{defn}
\begin{rem}
In the special case when $K$ is the simplicial complex associated to a Delzant polytope $\Delta$, $\calZ_{K}/S$ is
homeomorphic to the corresponding symplectic toric manifold. We have $\Tor_i^{\ZZ[\sfR^*]}(\ZZ[K],\ZZ) = 0$ for all
$i\geq 1$ since we know that $\ZZ[K]$ is free over $\ZZ[\sfR^*]$. Thus the long exact sequence (\ref{GS}) implies that
$\Tor_i^{\ZZ[\tilde{R}^*]}(\ZZ[K],\ZZ) = 0$ for all $i \geq 2$. Hence the only non-zero part of the long exact sequence
is:
\[
\xymatrix{ 0\ar[r]&
  \Tor_{1}^{\ZZ[\tilde{\sfR}^*]}(\ZZ[K],\ZZ) \ar[r]^{\tau_{*}} &
  \Tor_{0}^{\ZZ[\sfR^*]}(\ZZ[K],\ZZ) \ar[r]^{\delta} &
  \Tor_{0}^{\ZZ[\sfR^*]}(\ZZ[K],\ZZ) \ar[r]^{\tau^{*}} &
   \Tor_{0}^{\ZZ[\tilde{\sfR}^*]}(\ZZ[K],\ZZ) \ar[r]^{} & 0
    }
\]
This sequence, together with the identification of torsion algebras and cohomology rings of toric manifolds, gives the Gysin sequence used in Luo's paper \cite{Luo10} to compute the cohomology ring of a good contact toric manifold.
\end{rem}


\section*{Acknowledgements}
The authors want to thank M. Franz, T. Holm, Y. Karshon, A. Knutson, T. Ohmoto, K. Ono,  D. Suh for important advice and useful conversations. The first author is particularly indebted to K. Ono for providing him an excellent environment at Hokkaido University where he had spent significant time for this paper in July and August 2011. The first author would like to show his gratitude to the Algebraic Structure and its Application Research Center (ASARC) at KAIST for its constant support. The first author is also supported by the National Research Foundation of Korea (NRF) grant funded by the Korea government (MEST) (No. 2012-0000795, 2011-0001181).
\bibliography{references}{}

\begin{thebibliography}{10}

\bibitem{ALR}
{\sc Adem, A., Leida, J., and Ruan, Y.}
\newblock {\em Orbifolds and stringy topology}, vol.~171 of {\em Cambridge
  Tracts in Mathematics}.
\newblock Cambridge University Press, Cambridge, 2007.

\bibitem{Allday}
{\sc Allday, C.}
\newblock A family of unusual torus group actions.
\newblock In {\em Group actions on manifolds ({B}oulder, {C}olo., 1983)},
  vol.~36 of {\em Contemp. Math.} Amer. Math. Soc., Providence, RI, 1985,
  pp.~107--111.

\bibitem{BBP04}
{\sc Baskakov, I.~V., Bukhshtaber, V.~M., and Panov, T.~E.}
\newblock Algebras of cellular cochains, and torus actions.
\newblock {\em Uspekhi Mat. Nauk 59}, 3(357) (2004), 159--160.

\bibitem{BCS}
{\sc Borisov, L.~A., Chen, L., and Smith, G.~G.}
\newblock The orbifold {C}how ring of toric {D}eligne-{M}umford stacks.
\newblock {\em J. Amer. Math. Soc. 18}, 1 (2005), 193--215 (electronic).

\bibitem{BHcomm}
{\sc Bruns, W., and Herzog, J.}
\newblock {\em Cohen-{M}acaulay rings}, vol.~39 of {\em Cambridge Studies in
  Advanced Mathematics}.
\newblock Cambridge University Press, Cambridge, 1993.

\bibitem{BP}
{\sc Buchstaber, V.~M., and Panov, T.~E.}
\newblock {\em Torus actions and their applications in topology and
  combinatorics}, vol.~24 of {\em University Lecture Series}.
\newblock American Mathematical Society, Providence, RI, 2002.

\bibitem{BP99}
{\sc Bukhshtaber, V.~M., and Panov, T.~E.}
\newblock Torus actions and the combinatorics of polytopes.
\newblock {\em Tr. Mat. Inst. Steklova 225}, Solitony Geom. Topol. na
  Perekrest. (1999), 96--131.

\bibitem{Cox95}
{\sc Cox, D.~A.}
\newblock The homogeneous coordinate ring of a toric variety.
\newblock {\em J. Algebraic Geom. 4}, 1 (1995), 17--50.

\bibitem{Danilov78}
{\sc Danilov, V.~I.}
\newblock The geometry of toric varieties.
\newblock {\em Uspekhi Mat. Nauk 33}, 2(200) (1978), 85--134, 247.

\bibitem{DavisJanuszkiewicz91}
{\sc Davis, M.~W., and Januszkiewicz, T.}
\newblock Convex polytopes, {C}oxeter orbifolds and torus actions.
\newblock {\em Duke Math. J. 62}, 2 (1991), 417--451.

\bibitem{Edidin10}
{\sc Edidin, D.}
\newblock Equivariant geometry and the cohomology of the moduli space of
  curves.
\newblock \texttt{arXiv:1006.2364}.

\bibitem{EdidinGraham98}
{\sc Edidin, D., and Graham, W.}
\newblock Equivariant intersection theory.
\newblock {\em Invent. Math. 131}, 3 (1998), 595--634.

\bibitem{Franz02}
{\sc Franz, M.}
\newblock Koszul duality and equivariant cohomology for tori.
\newblock {\em Int. Math. Res. Not.}, 42 (2003), 2255--2303.

\bibitem{Franz05i}
{\sc Franz, M.}
\newblock The integral cohomology of toric manifolds.
\newblock {\em Tr. Mat. Inst. Steklova 252}, Geom. Topol., Diskret. Geom. i
  Teor. Mnozh. (2006), 61--70.

\bibitem{Franz09}
{\sc Franz, M.}
\newblock Describing toric varieties and their equivariant cohomology.
\newblock {\em Colloq. Math. 121}, 1 (2010), 1--16.

\bibitem{FranzPuppe07}
{\sc Franz, M., and Puppe, V.}
\newblock Exact cohomology sequences with integral coefficients for torus
  actions.
\newblock {\em Transform. Groups 12}, 1 (2007), 65--76.

\bibitem{FranzPuppeOsaka}
{\sc Franz, M., and Puppe, V.}
\newblock Freeness of equivariant cohomology and mutants of compactified
  representations.
\newblock In {\em Toric topology}, vol.~460 of {\em Contemp. Math.} Amer. Math.
  Soc., Providence, RI, 2008, pp.~87--98.

\bibitem{HatcherAT}
{\sc Hatcher, A.}
\newblock {\em Algebraic topology}.
\newblock Cambridge University Press, Cambridge, 2002.

\bibitem{HM}
{\sc Holm, T., and Matsumura, T.}
\newblock Equivariant cohomology for hamiltonian torus actions on symplectic
  orbifolds.
\newblock To Appear in \emph{Transformation Groups}, \texttt{arXiv:1008.3315}.

\bibitem{Holm08}
{\sc Holm, T.~S.}
\newblock Orbifold cohomology of abelian symplectic reductions and the case of
  weighted projective spaces.
\newblock In {\em Poisson geometry in mathematics and physics}, vol.~450 of
  {\em Contemp. Math.} Amer. Math. Soc., Providence, RI, 2008, pp.~127--146.

\bibitem{Jurkiewicz85}
{\sc Jurkiewicz, J.}
\newblock Torus embeddings, polyhedra, {$k^\ast$}-actions and homology.
\newblock {\em Dissertationes Math. (Rozprawy Mat.) 236\/} (1985), 64.

\bibitem{LerMal}
{\sc Lerman, E., and Malkin, A.}
\newblock Hamiltonian group actions on symplectic {D}eligne-{M}umford stacks
  and toric orbifolds.
\newblock {\em Adv. Math. 229}, 2 (2012), 984--1000.

\bibitem{LT}
{\sc Lerman, E., and Tolman, S.}
\newblock Hamiltonian torus actions on symplectic orbifolds and toric
  varieties.
\newblock {\em Trans. Amer. Math. Soc. 349}, 10 (1997), 4201--4230.

\bibitem{Luo10}
{\sc Luo, S.}
\newblock Cohomology rings of good contact toric manifolds.
\newblock \texttt{arXiv:1012.2146}.

\bibitem{MatsumuraComm}
{\sc Matsumura, H.}
\newblock {\em Commutative algebra}, second~ed., vol.~56 of {\em Mathematics
  Lecture Note Series}.
\newblock Benjamin/Cummings Publishing Co., Inc., Reading, Mass., 1980.

\bibitem{McCleary}
{\sc McCleary, J.}
\newblock {\em A user's guide to spectral sequences}, second~ed., vol.~58 of
  {\em Cambridge Studies in Advanced Mathematics}.
\newblock Cambridge University Press, Cambridge, 2001.

\bibitem{PoddarSarkar10}
{\sc Poddar, M., and Sarkar, S.}
\newblock On quasitoric orbifolds.
\newblock {\em Osaka J. Math. 47}, 4 (2010), 1055--1076.

\bibitem{Rom}
{\sc Romagny, M.}
\newblock Group actions on stacks and applications.
\newblock {\em Michigan Math. J. 53}, 1 (2005), 209--236.

\bibitem{SerreLocal}
{\sc Serre, J.-P.}
\newblock {\em Local algebra}.
\newblock Springer Monographs in Mathematics. Springer-Verlag, Berlin, 2000.
\newblock Translated from the French by CheeWhye Chin and revised by the
  author.

\bibitem{TolmanThesis}
{\sc Tolman, S.}
\newblock {\em Group Actions And Cohomology}.
\newblock {PhD} dissertation, Harvard University, Dept. Math., 1993.

\end{thebibliography}
\bibliographystyle{acm}
\end{document}